\documentclass[10pt]{article}

\usepackage[utf8]{inputenc}
\usepackage[T1]{fontenc}

\usepackage{epsf}
\usepackage{amsmath}

\allowdisplaybreaks

\usepackage[showframe=false]{geometry}
\usepackage{changepage}

\usepackage{epsfig}
\usepackage{amssymb}

\usepackage{amsthm}
\usepackage{setspace}
\usepackage{cite}
\usepackage{mcite}

\usepackage{algorithmic}  
\usepackage{algorithm}

\usepackage{shadow}
\usepackage{fancybox}
\usepackage{fancyhdr}

\usepackage{color}
\usepackage[usenames,dvipsnames,svgnames,table]{xcolor}
\newcommand{\bl}[1]{\textcolor{blue}{#1}}

\definecolor{mypurple}{rgb}{.4,.0,.5}

\usepackage[hyphens]{url}

\usepackage[colorlinks=true,
            linkcolor=black,
            urlcolor=blue,
            citecolor=purple]{hyperref}

\usepackage{breakurl}

\def\y{{\bf y}}

\def\x{{\bf x}}

\def\x{{\mathbf x}}

\def\u{{\bf u}}

\def\x{{\bf x}}
\def\y{{\bf y}}

\def\q{{\bf q}}
\def\m{{\bf m}}

\def\c{{\bf c}}

\def\h{{\bf h}}

\def\cH{{\mathcal H}}

\def\be{\begin{equation}}
\def\ee{\end{equation}}
\def\ba{\left[\begin{array}}
\def\ea{\end{array}\right]}

\def\u{{\bf u}}

\def\x{{\bf x}}
\def\y{{\bf y}}

\def\q{{\bf q}}

\def\c{{\bf c}}

\def\p{{\bf p}}

\def\1{{\bf 1}}

\def\0{{\bf 0}}

\def\calX{{\cal X}}
\def\calY{{\cal Y}}

\def\mR{{\mathbb R}}
\def\mN{{\mathbb N}}
\def\mE{{\mathbb E}}
\def\mS{{\mathbb S}}

\def\lp{\left (}
\def\rp{\right )}

\def\y{{\bf y}}

\def\x{{\bf x}}

\def\x{{\mathbf x}}

\def\u{{\bf u}}

\def\x{{\bf x}}
\def\y{{\bf y}}

\def\q{{\bf q}}

\def\c{{\bf c}}

\def\h{{\bf h}}

\def\cH{{\cal H}}

\def\be{\begin{equation}}
\def\ee{\end{equation}}
\def\ba{\left[\begin{array}}
\def\ea{\end{array}\right]}

\def\u{{\bf u}}

\def\x{{\bf x}}
\def\y{{\bf y}}

\def\q{{\bf q}}

\def\c{{\bf c}}

\def\p{{\bf p}}

\def\({\left (}
\def\){\right )}

\def\1{{\bf 1}}
\def\m{{\bf m}}
\def\q{{\bf q}}

\def\0{{\bf 0}}

\def\cX{{\mathcal X}}
\def\cY{{\mathcal Y}}

\usepackage{xcolor}
\usepackage{color}

\definecolor{darkgreen}{rgb}{0, 0.4,0}

\definecolor{purplebrown}{rgb}{0.5,0.1,0.6}

\definecolor{ultclupcol}{rgb}{0.1,0.5,0.5}

\definecolor{mytrycolor}{rgb}{0.5,0.7,0.2}

\definecolor{ultclupcola}{rgb}{.5,0,.5}

\definecolor{shadebrown}{rgb}{0.1,0.1,0.9}
\definecolor{lightblue}{rgb}{0.2,0,1}

\usepackage{fancybox}
\usepackage{graphicx}
\usepackage{epstopdf}
\usepackage{epsfig}
\usepackage{wrapfig}
\usepackage{subfigure}

\usepackage{xcolor}
\usepackage{tcolorbox}
\tcbuselibrary{skins}

\newtcbox{\xmybox}{on line,
arc=7pt,
before upper={\rule[-3pt]{0pt}{10pt}},boxrule=0pt,
boxsep=0pt,left=6pt,right=6pt,top=0pt,bottom=0pt,enhanced, coltext=blue, colback=white!10!yellow}

\newtcbox{\xmyboxa}{on line,
arc=7pt,
before upper={\rule[-3pt]{0pt}{10pt}},boxrule=0pt,
boxsep=0pt,left=6pt,right=6pt,top=0pt,bottom=0pt,enhanced, colback=white!10!yellow}

\newtcbox{\xmyboxb}{on line,
arc=7pt,
before upper={\rule[-3pt]{0pt}{10pt}},boxrule=1pt,colframe=darkgreen!100!blue,
boxsep=0pt,left=6pt,right=6pt,top=0pt,bottom=0pt,enhanced, colback=white!10!yellow}

\newtcbox{\xmyboxc}{on line,
arc=7pt,
before upper={\rule[-3pt]{0pt}{10pt}},boxrule=.7pt,colframe=blue!100!blue,
boxsep=0pt,left=6pt,right=6pt,top=0pt,bottom=0pt,enhanced, coltext=blue, colback=white!10!yellow}

\newtcbox{\xmytboxa}{on line,
arc=7pt,
before upper={\rule[-3pt]{0pt}{10pt}},boxrule=.0pt,colframe=pink!50!yellow,
boxsep=0pt,left=6pt,right=6pt,top=0pt,bottom=0pt,enhanced, coltext=white, colback=blue!40!red}

\newtcbox{\xmytboxb}{on line,
arc=7pt,
before upper={\rule[-3pt]{0pt}{10pt}},boxrule=.0pt,colframe=pink!50!yellow,
boxsep=0pt,left=6pt,right=6pt,top=0pt,bottom=0pt,enhanced, coltext=white, colback=white!40!green}

\makeatletter
\newcommand\subsubsubsection{\@startsection{paragraph}{4}{\z@}{-2.5ex\@plus -1ex \@minus -.25ex}{1.25ex \@plus .25ex}{\normalfont\normalsize\bfseries}}
\newcommand\subsubsubsubsection{\@startsection{subparagraph}{5}{\z@}{-2.5ex\@plus -1ex \@minus -.25ex}{1.25ex \@plus .25ex}{\normalfont\normalsize\bfseries}}
\makeatother

\newtheorem{theorem}{Theorem}

\newtheorem{corollary}{Corollary}

\begin{document}

\begin{singlespace}

\title {Studying Hopfield models via fully lifted random duality theory 
}
\author{
\textsc{Mihailo Stojnic
\footnote{e-mail: {\tt flatoyer@gmail.com}} }}
\date{}
\maketitle

\centerline{{\bf Abstract}} \vspace*{0.1in}

Relying on a recent progress made in studying bilinearly indexed (bli) random processes in \cite{Stojnicnflgscompyx23,Stojnicsflgscompyx23}, the main foundational principles of fully lifted random duality theory (fl RDT) were established in \cite{Stojnicflrdt23}. We here study famous Hopfield models and show that their statistical behavior can be characterized via the fl RDT. Due to a nestedly lifted nature, the resulting characterizations and, therefore, the whole analytical machinery that produces them, become fully operational only if one can successfully conduct underlying numerical evaluations. After conducting such evaluations for both positive and negative Hopfield models, we observe a remarkably fast convergence of the fl RDT mechanism. Namely, for the so-called square case, the fourth decimal precision is achieved already on the third (second non-trivial) level of lifting (3-sfl RDT) for the positive  and on the fourth (third non-trivial) level of lifting (4-sfl RDT) for the corresponding negative model. In particular,  we obtain the scaled ground state free energy $\approx 1.7788$ for the positive and  $\approx 0.3279$ for the negative model.

\vspace*{0.25in} \noindent {\bf Index Terms: Hopfield models; Fully lifted random duality theory}.

\end{singlespace}

\section{Introduction}
\label{sec:back}

Many natural phenomena, typically uncovered and explained through the postulates of statistical and experimental physics, have become  very popular subjects of intensive studies in a host of mathematical fields over the last several decades. Development of a synergistic approach, where key physical insights are used to trace paths for mathematically rigorous considerations, turned out to be very fruitful. Among the most successful examples of such a synergy are certainly the considerations related to the physics of the matter structure and, in particular, those related to the so called Sherrington-Kirkpatrick (SK) models (see, e.g., \cite{Par79,Par80,Par83,Parisi80,SheKir72} and \cite{Talbook11a,Talbook11b,Pan10,Pan10a,Pan13,Pan13a,Pan14,Guerra03,Tal06}).

As was expected (or even predicted), such synergistic approach evolved even further over time. Understanding the accompanying mathematics became more complete and, consequently, one reached the level where many of the relevant random structures, not even known to have any physical meaning, could be rigorously handled by following the statistical physics ideas as key insights. The number of such ``\emph{artificial}'' examples (where the direct connection to the corresponding natural phenomena is not immediately apparent) quickly grew over the last several decades and nowadays, quite likely, dominates the number of the ``\emph{natural}'' ones.  Moreover, given the pace at which the new ones are appearing, one can be assured that many more are yet to come. What often makes them particularly attractive, is that they typically span a variety of scientific fields ranging from theoretical and applied mathematics and probability (see, e.g., \cite{Talbook11a,Talbook11b,Pan10,Pan10a,Pan13,Pan13a,Pan14,Guerra03,Tal06,StojnicRegRndDlt10,Stojnicgscompyx16,Stojnicgscomp16,SchTir03}),  to signal processing and information theory (see, e.g., \cite{BayMon10,BayMon10lasso,StojnicCSetam09,StojnicGenLasso10}), and reaching to theoretical/algorithmic  computer science
(see, e.g., \cite{Moll12,CojOgh14,DinSlySun15,MezParZec02,MerMezZec06,FriWor05,AchPer04,AlmSor02,Wast12,LinWa04,Talbook11a,Talbook11b,NaiPraSha05,CopSor02,Ald01}) and many others. Moreover, they also include a host of different topics within each of these areas. Examples are again rather diverse, ranging from spherical/binary perceptrons in neural networks and machine learning (see, e.g., \cite{SchTir03,StojnicGardGen13,Tal05,Talbook11a,Talbook11b,Wendel,Winder,Winder61,Cover65,DonTan09Univ}), to compressed sensing and statistical regression in signal processing and information theory (see, e.g., \cite{StojnicUpper10,StojnicCSetam09,StojnicGenLasso10,BayMon10,BayMon10lasso,DonohoPol}), and satisfiability/assignment/matching problems in theoretical computer science and optimization and algorithms theory
(see, e.g., \cite{CojOgh14,DinSlySun14,DinSlySun15,DinSlySun14,MezParZec02,MerMezZec06,FriWor05,AchPer04}) and \cite{AlmSor02,Wast04,Wast09,Wast12,LinWa04,TalAssignment03,Talbook11a,Talbook11b,NaiPraSha05,CopSor02,CopSor99,Ald01,Ald92,MezPar87}).

Here, we are interested in yet another particular type of random structure, the so-called Hopfied models, that have gained a lot interest and, consequently, have been extensively studied in many different scientific fields (see, e.g.,
\cite{Hop82,PasFig78,Hebb49,PasShchTir94,ShchTir93,BarGenGueTan10,BarGenGueTan12,Tal98,StojnicMoreSophHopBnds10,BovGay98,Zhao11,Talbook11a,Talbook11b}). These models ware popularized in \cite{Hop82} (or, if viewed in different contexts, one can say even in \cite{PasFig78,Hebb49}). To provide their a mathematical description, one starts with the so-called Hamiltonian of the following type
\begin{equation}
\cH(G)=\sum_{i\neq j}^{n}  A_{ij}\x_i\x_j,\label{eq:ham}
\end{equation}
where
\begin{equation}
A_{ij}(G)=\sum_{l=1}^{m}  G_{il}G_{lj},\label{eq:hamAij}
\end{equation}
are the so-called quenched interactions and $G$ is an $m\times n$ matrix. For the mathematical convenience, we then combine (\ref{eq:ham}) and (\ref{eq:hamAij}) to obtain a more compact form
\begin{equation}
\cH(G)= \x^TG^TG\x,\label{eq:a0ham1}
\end{equation}
where we include the diagonal elements as well (as it will become clear throughout the presentation, such an inclusion contributes to the easiness of writing and the overall clarity while affecting, exactly, in no way the ensuing analytical considerations). Depending on the context where the models are used, the components of $G$ play more or less prominent role on their own (in other words, they are relevant by themselves rather than just as constitutive components of $A$). For example, within the neural networks contexts, one can often view $G$ as the matrix of the so-called stored patterns. On the other hand, we, in this paper view $G$ as a purely mathematical object. While the overall treatment is on mathematical objects level, we, often, refer to many relevant quantities utilizing the terminology frequently seen in statistical physics. Along the same lines, our prevalent interest is in the so-called thermodynamic limit, or, as it is called in mathematics, large $n$ linear regime. In other words, we are interested in scenarios where $n\rightarrow\infty$, and all other dimensions are linearly proportional to $n$ with proportionalities remaining constant as $n$ grows. For example,
we assume that $m$ is also large and $\alpha=\lim_{n\rightarrow\infty}\frac{m}{n}$ with $\alpha$ remaining constant as $n$ grows. The nature of the components of $G$ also fluctuates throughout the literature, again, depending on the context where the models are used. For example, in neural nets, the elements are typically binary spins, whereas in other fields their various continuous statistical forms are often seen as well. Here, we view them as i.i.d. standard normals, but the results hold for pretty much any statistics that can be pushed through the central limit theorem.

To characterize the behavior of the physical interpretations described through the above Hamiltonian, one usually looks at the partition function
\begin{equation}
Z^{\pm}(\beta,G)=\sum_{\x\in\cX}e^{\pm\beta\cH(G)},\label{eq:a0partfun}
\end{equation}
where $\beta>0$ is a parameter, typically called the \emph{inverse temperature}, and $\cX$ is a set typically consisting of the corners of the unit norm $n$-dimensional hypercube, i.e. $\cX=\{-\frac{1}{\sqrt{n}},\frac{1}{\sqrt{n}}\}^n$ (from the mathematical point of view though, $\cX$ can be tretated as actually any set). We adopt the convention of \cite{DeanRit01,StojnicMoreSophHopBnds10} and refer to the sceanario with ``$+$'' sign as the \emph{positive} Hopfield model, and to the one with ``$-$'' sign as the \emph{negative} Hopfiled model. Instead of directly handling the partition function, $Z^{\pm}(\beta,G)$, one often considers the following thermodynamic limit, scaled $\log$, version called (average) free energy
\begin{equation}
f^{\pm}(\beta)=\lim_{n\rightarrow\infty}\frac{\mE_G\log{(Z^{\pm}(\beta,G)})}{\beta n}
=\lim_{n\rightarrow\infty} \frac{\mE_G\log{(\sum_{\x\in\cX} e^{\pm\beta\cH(G)})}}{\beta n},\label{eq:a0logpartfunsqrt}
\end{equation}
where $\mE_G$ stands for the expectation with respect to $G$ (throughout the paper, $\mE$ always stands for the expectation and its subscript indicates the type of randomness with respect to which the expectation is taken). Of particular interest is often
the thermodynamic limit ground state regime, where, in addition to $n\rightarrow\infty$, one also has
$\beta\rightarrow\infty$ and
\begin{eqnarray}
f^{\pm}(\infty)   \triangleq  \lim_{\beta\rightarrow\infty}f^{\pm}(\beta) & = &
\lim_{\beta,n\rightarrow\infty}\frac{\mE_G\log{(Z^{\pm}(\beta,G)})}{\beta n}=
 \lim_{n\rightarrow\infty}\frac{\mE_G \max_{\x\in\cX}\pm \cH(G)}{n} \nonumber \\
 & = &\lim_{n\rightarrow\infty}\frac{\mE_G \max_{\x\in\cX}\pm \x^TG^TG\x}{n}.
  \label{eq:a0limlogpartfunsqrt}
\end{eqnarray}
The above corresponds to the standard Hopfield model treatment. As mentioned earlier, these models have a long history and have been studied extensively throughout the literature in a variety of fields (see, e.g., excellent references \cite{PasShchTir94,ShchTir93,BarGenGueTan10,BarGenGueTan12,Tal98,CriAmiGut86,SteKuh94,AmiGutSom87}).
Here, we find it convenient to study their so-called \emph{square-root} counterparts (see, e.g., \cite{StojnicMoreSophHopBnds10}). For such variants  all of the above remains in place with the exception that, instead of $\cH$ from (\ref{eq:a0ham1}), one now considers
\begin{equation}
\cH_{sq}(G)= \y^TG\x,\label{eq:ham1}
\end{equation}
and the corresponding partition function
\begin{equation}
Z_{sq}^{\pm}(\beta,G)=\sum_{\x\in\cX} \lp \sum_{\y\in\cY}e^{\beta\cH_{sq}(G)}\rp^{\pm1}.\label{eq:partfun}
\end{equation}
where $\cY=\mS^m$ (with $\mS^m$ being the $m$-dimensional unit sphere). The thermodynamic limit (average) free energy becomes
\begin{eqnarray}
f_{sq}^{\pm}(\beta) & = & \lim_{n\rightarrow\infty}\frac{\mE_G\log{(Z_{sq}^{\pm}(\beta,G)})}{\beta \sqrt{n}}
=\lim_{n\rightarrow\infty} \frac{\mE_G\log\lp \sum_{\x\in\cX} \lp \sum_{\y\in\cY}e^{\beta\cH_{sq}(G)}\rp^{\pm1}\rp}{\beta \sqrt{n}} \nonumber \\
& = &\lim_{n\rightarrow\infty} \frac{\mE_G\log\lp \sum_{\x\in\cX} \lp \sum_{\y\in\cY}e^{\beta\y^TG\x)}\rp^{\pm1}\rp}{\beta \sqrt{n}},\label{eq:logpartfunsqrt}
\end{eqnarray}
and its corresponding ground state special case
\begin{eqnarray}
f_{sq}^{\pm}(\infty)   \triangleq    \lim_{\beta\rightarrow\infty}f_{sq}^{\pm}(\beta) & = &
\lim_{\beta,n\rightarrow\infty}\frac{\mE_G\log{(Z_{sq}^{\pm}(\beta,G)})}{\beta \sqrt{n}}=
 \lim_{n\rightarrow\infty}\frac{\mE_G \max_{\x\in\cX}  \pm  \max_{\y\in\cY} \y^TG\x}{\sqrt{n}} \nonumber \\
& = &  \lim_{n\rightarrow\infty}\frac{\mE_G \max_{\x\in\cX}\pm \sqrt{\x^TG^TG\x}}{\sqrt{n}}.
  \label{eq:limlogpartfunsqrt}
\end{eqnarray}

Our main focus throughput the paper is precisely the free energy given in (\ref{eq:logpartfunsqrt}). However, as the  prevalent interest is to properly understand the, mathematically most challenging, ground state regime, we use the studying of the free energy, defined for general $\beta$,  to eventually deduce particular ground state behavior (\ref{eq:limlogpartfunsqrt}). Along the same lines, while the analysis  presented below accounts for any $\beta$, in the interest of easing the exposition, we, on occasion, neglect some of the terms with no significant role in the ground state regime.

While we study the square-root variant of the model, the free energy concentrations ensure that the obtained results immediately translate to the standard model as well. It is also important to note that, despite the fact that Hopfield models have been studied for several decades, analytical results particularly related to the problems of our interests here are rather scarce. For example, \cite{CriAmiGut86,SteKuh94,AmiGutSom87} cover the range of all $\beta$'s and offer a statistical physics type of treatment with replica analysis based estimates. On the other hand, mathematically rigorous treatments (see, e.g., \cite{PasShchTir94,ShchTir93,BarGenGueTan10,Tal98,Zhao11,BarGenGue11bip}), are usually restricted to certain $\alpha$ and/or $\beta$ regimes. Most of them typically relate to the high-temperature ($T$) regime (low inverse temperature $\beta=\frac{1}{T}$), where the so-called replica symmetric behavior is in place. The regimes, where the replica symmetry is expected not to hold, are, on the other hand, much harder. Moreover, the hardest ones among them are precisely the ground state ones of interest here. They assume zero temperature scenarios ($\beta\rightarrow\infty$) and  the available results for such regimes are even scarcer. In particular, utilizing the random duality theory (RDT) machinery, \cite{StojnicHopBnds10} established free energy upper bounds which also turned out to precisely match the replica symmetry based predictions. \cite{StojnicMoreSophHopBnds10} then went a bit further, used a \emph{lifted} RDT variant, and lowered the bounds of \cite{StojnicHopBnds10}, thereby showing that the replica symmetry predictions are \emph{not} tight and signaling that the assumed symmetry \emph{must} be broken.

We here first recognize the connection of the underlying statistical Hopfield models and bilinearly indexed (bli) random processes. Utilizing a recent progress made in studying such processes in \cite{Stojnicsflgscompyx23,Stojnicnflgscompyx23}, in \cite{Stojnicflrdt23} a \emph{fully lifted} random duality theory (fl RDT) was established. Keeping that in mind, we first observe that the square-root Hopfield models can be characterized via the fl RDT and, then, use a particular \emph{stationarized} fl RDT variant (called sfl RDT) to obtain desired characterization.

\section{Connecting Hopfield models and sfl RDT}
\label{sec:randlincons}

We start by observing that the free energy from (\ref{eq:logpartfunsqrt}),
\begin{eqnarray}
f_{sq}^{\pm}(\beta) & = &\lim_{n\rightarrow\infty} \frac{\mE_G\log\lp \sum_{\x\in\cX} \lp \sum_{\y\in\cY}e^{\beta\y^TG\x)}\rp^{\pm1}\rp}{\beta \sqrt{n}},\label{eq:hmsfl1}
\end{eqnarray}
is essentially a function of bli $\y^TG\x$. To make a connection with the bli related results of \cite{Stojnicsflgscompyx23,Stojnicnflgscompyx23}, we need a few definitions. For $r\in\mN$, $k\in\{1,2,\dots,r+1\}$, real scalars $s$, $x$, and $y$  such that $s^2=1$, $x>0$, and $y>0$, sets $\cX\subseteq \mR^n$ and $\cY\subseteq \mR^m$, function $f_S(\cdot):\mR^n\rightarrow R$, vectors $\p=[\p_0,\p_1,\dots,\p_{r+1}]$, $\q=[\q_0,\q_1,\dots,\q_{r+1}]$, and $\c=[\c_0,\c_1,\dots,\c_{r+1}]$ such that
 \begin{eqnarray}\label{eq:hmsfl2}
1=\p_0\geq \p_1\geq \p_2\geq \dots \geq \p_r\geq \p_{r+1} & = & 0 \nonumber \\
1=\q_0\geq \q_1\geq \q_2\geq \dots \geq \q_r\geq \q_{r+1} & = &  0,
 \end{eqnarray}
$\c_0=1$, $\c_{r+1}=0$, and ${\mathcal U}_k\triangleq [u^{(4,k)},\u^{(2,k)},\h^{(k)}]$  such that the components of  $u^{(4,k)}\in\mR$, $\u^{(2,k)}\in\mR^m$, and $\h^{(k)}\in\mR^n$ are i.i.d. standard normals, we set
  \begin{eqnarray}\label{eq:fl4}
\psi_{S,\infty}(f_{S},\calX,\calY,\p,\q,\c,x,y,s)  =
 \mE_{G,{\mathcal U}_{r+1}} \frac{1}{n\c_r} \log
\lp \mE_{{\mathcal U}_{r}} \lp \dots \lp \mE_{{\mathcal U}_2}\lp\lp\mE_{{\mathcal U}_1} \lp \lp Z_{S,\infty}\rp^{\c_2}\rp\rp^{\frac{\c_3}{\c_2}}\rp\rp^{\frac{\c_4}{\c_3}} \dots \rp^{\frac{\c_{r}}{\c_{r-1}}}\rp, \nonumber \\
 \end{eqnarray}
where
\begin{eqnarray}\label{eq:fl5}
Z_{S,\infty} & \triangleq & e^{D_{0,S,\infty}} \nonumber \\
 D_{0,S,\infty} & \triangleq  & \max_{\x\in\cX,\|\x\|_2=x} s \max_{\y\in\cY,\|\y\|_2=y}
 \lp f_{S}
+\sqrt{n}  y    \lp\sum_{k=2}^{r+1}c_k\h^{(k)}\rp^T\x
+ \sqrt{n} x \y^T\lp\sum_{k=2}^{r+1}b_k\u^{(2,k)}\rp \rp \nonumber  \\
 b_k & \triangleq & b_k(\p,\q)=\sqrt{\p_{k-1}-\p_k} \nonumber \\
c_k & \triangleq & c_k(\p,\q)=\sqrt{\q_{k-1}-\q_k}.
 \end{eqnarray}
 The  following theorem is a fundamental element of sfl RDT.
\begin{theorem} [\cite{Stojnicflrdt23}]
\label{thm:thmsflrdt1}  Assume that $n$ is large and that $\alpha=\lim_{n\rightarrow\infty} \frac{m}{n}$, where $\alpha$ remains constant as  $n$ grows. Let matrix $G\in\mR^{m\times n}$
 be comprised of i.i.d. standard normal elements and let $\cX\subseteq \mR^n$ and $\cY\subseteq \mR^m$ be two given sets. Assume the complete sfl RDT frame from \cite{Stojnicsflgscompyx23} and consider a given function $f(\x):R^n\rightarrow R$. Set
\begin{align}\label{eq:thmsflrdt2eq1}
   \psi_{rp} & \triangleq - \max_{\x\in\cX} s \max_{\y\in\cY} \lp f(\x)+\y^TG\x \rp
   \qquad  \mbox{(\bl{\textbf{random primal}})} \nonumber \\
   \psi_{rd}(\p,\q,\c,x,y,s) & \triangleq    \frac{x^2y^2}{2}    \sum_{k=2}^{r+1}\Bigg(\Bigg.
   \p_{k-1}\q_{k-1}
   -\p_{k}\q_{k}
  \Bigg.\Bigg)
\c_k
  - \psi_{S,\infty}(f(\x),\calX,\calY,\p,\q,\c,x,y,s) \hspace{.04in} \mbox{(\bl{\textbf{fl random dual}})}. \nonumber \\
 \end{align}
Let $\hat{\p_0}\rightarrow 1$, $\hat{\q_0}\rightarrow 1$, and $\hat{\c_0}\rightarrow 1$, $\hat{\p}_{r+1}=\hat{\q}_{r+1}=\hat{\c}_{r+1}=0$, and let the non-fixed parts of $\hat{\p}\triangleq \hat{\p}(x,y)$, $\hat{\q}\triangleq \hat{\q}(x,y)$, and  $\hat{\c}\triangleq \hat{\c}(x,y)$ be the solutions of the following system
\begin{eqnarray}\label{eq:thmsflrdt2eq2}
   \frac{d \psi_{rd}(\p,\q,\c,x,y,s)}{d\p} =  0,\quad
   \frac{d \psi_{rd}(\p,\q,\c,x,y,s)}{d\q} =  0,\quad
   \frac{d \psi_{rd}(\p,\q,\c,x,y,s)}{d\c} =  0.
 \end{eqnarray}
 Then,
\begin{eqnarray}\label{eq:thmsflrdt2eq3}
    \lim_{n\rightarrow\infty} \frac{\mE_G  \psi_{rp}}{\sqrt{n}}
  & = &
\min_{x>0} \max_{y>0} \lim_{n\rightarrow\infty} \psi_{rd}(\hat{\p}(x,y),\hat{\q}(x,y),\hat{\c}(x,y),x,y,s) \qquad \mbox{(\bl{\textbf{strong sfl random duality}})},\nonumber \\
 \end{eqnarray}
where $\psi_{S,\infty}(\cdot)$ is as in (\ref{eq:fl4})-(\ref{eq:fl5}).
 \end{theorem}
\begin{proof}
The $s=-1$ scenario was proven in \cite{Stojnicflrdt23}. The $s=1$ follows through a line-by-line repetition and trivial adjustment of the arguments of Section 3 of \cite{Stojnicflrdt23} with $s=-1$ replaced by $s=1$.
 \end{proof}
The following corollary is the operational version of the above theorem directly applicable to the Hopfield models of interest here.
\begin{corollary}
\label{cor:cor1}  Assume the setup of Theorem \ref{thm:thmsflrdt1}. Let all elements of $\cX\subseteq \mR^n$ and $\cY\subseteq \mR^m$ have unit norm. Set
\begin{align}\label{eq:thmsflrdt2eq1a0}
   \psi_{rp} & \triangleq - \max_{\x\in\cX} s \max_{\y\in\cY} \lp\y^TG\x \rp
   \qquad  \mbox{(\bl{\textbf{random primal}})} \nonumber \\
   \psi_{rd}(\p,\q,\c,x,y,s) & \triangleq    \frac{1}{2}    \sum_{k=2}^{r+1}\Bigg(\Bigg.
   \p_{k-1}\q_{k-1}
   -\p_{k}\q_{k}
  \Bigg.\Bigg)
\c_k
  - \psi_{S,\infty}(0,\calX,\calY,\p,\q,\c,1,1,s) \qquad \mbox{(\bl{\textbf{fl random dual}})}. \nonumber \\
 \end{align}
Let the non-fixed parts of $\hat{\p}$, $\hat{\q}$, and  $\hat{\c}$ be the solutions of the following system
\begin{eqnarray}\label{eq:thmsflrdt2eq2a0}
   \frac{d \psi_{rd}(\p,\q,\c,1,1,s)}{d\p} =  0,\quad
   \frac{d \psi_{rd}(\p,\q,\c,1,1,s)}{d\q} =  0,\quad
   \frac{d \psi_{rd}(\p,\q,\c,1,1,s)}{d\c} =  0.
 \end{eqnarray}
 Then,
\begin{eqnarray}\label{eq:thmsflrdt2eq3a0}
    \lim_{n\rightarrow\infty} \frac{\mE_G  \psi_{rp}}{\sqrt{n}}
  & = &
 \lim_{n\rightarrow\infty} \psi_{rd}(\hat{\p},\hat{\q},\hat{\c},1,1,s) \qquad \mbox{(\bl{\textbf{strong sfl random duality}})},\nonumber \\
 \end{eqnarray}
where $\psi_{S,\infty}(\cdot)$ is as in (\ref{eq:fl4})-(\ref{eq:fl5}).
 \end{corollary}
\begin{proof}
A trivial direct consequence of Theorem \ref{thm:thmsflrdt1}, after taking $f(\x)=0$ and recognizing that sets $\cX$ and $\cY$ now have unit norm elements.
 \end{proof}

As noted in \cite{Stojnicflrdt23}, the above random primal problems' trivial concentrations enable various corresponding probabilistic variants of (\ref{eq:thmsflrdt2eq3}) and (\ref{eq:thmsflrdt2eq3a0}). We, however, skip stating such trivialities.

\section{Practical evaluations}
\label{sec:prac}

The results of Theorem \ref{thm:thmsflrdt1} and Corollary \ref{cor:cor1} have very elegant mathematical forms. Still, to see any formal utilization, one should be able to practically evaluate them. There are two main obstacles that one might face when trying to do so: (i) It is not, a priori, clear what should be the correct value for $r$. (ii) Set $\cY$ lacks a component-wise structure characterization and the decoupling over $\y$ is not necessarily overly obvious. For example, if the optimal $r$ is large (and the $r$-convergence of obtained results rather slow), or if the decoupling over $\y$ is technically complicated, the importance of the characterizations given in  Theorem \ref{thm:thmsflrdt1} and Corollary \ref{cor:cor1} would remain on a purely theoretical level and their elegance would prove as irrelevant regarding the practical usefulness. Below, we show that neither of these two obstacles poses a serious detriment to the overall success of the introduced methodology.

We start by recognizing that, based on the results of Corollary \ref{cor:cor1},  the key object of practical interest is the so-called \emph{random dual}
\begin{align}\label{eq:prac1}
    \psi_{rd}(\p,\q,\c,1,1,s) & \triangleq    \frac{1}{2}    \sum_{k=2}^{r+1}\Bigg(\Bigg.
   \p_{k-1}\q_{k-1}
   -\p_{k}\q_{k}
  \Bigg.\Bigg)
\c_k
  - \psi_{S,\infty}(0,\calX,\calY,\p,\q,\c,1,1,s). \nonumber \\
  & =   \frac{1}{2}    \sum_{k=2}^{r+1}\Bigg(\Bigg.
   \p_{k-1}\q_{k-1}
   -\p_{k}\q_{k}
  \Bigg.\Bigg)
\c_k
  - \frac{1}{n}\varphi(D^{(bin)}(s)) - \frac{1}{n}\varphi(D^{(sph)}(s)), \nonumber \\
  \end{align}
where analogously to (\ref{eq:fl4})-(\ref{eq:fl5})
  \begin{eqnarray}\label{eq:prac2}
\varphi(D,\c) & = &
 \mE_{G,{\mathcal U}_{r+1}} \frac{1}{\c_r} \log
\lp \mE_{{\mathcal U}_{r}} \lp \dots \lp \mE_{{\mathcal U}_3}\lp\lp\mE_{{\mathcal U}_2} \lp
\lp    e^{D}   \rp^{\c_2}\rp\rp^{\frac{\c_3}{\c_2}}\rp\rp^{\frac{\c_4}{\c_3}} \dots \rp^{\frac{\c_{r}}{\c_{r-1}}}\rp, \nonumber \\
  \end{eqnarray}
and
\begin{eqnarray}\label{eq:prac3}
D^{(bin)}(s) & = & \max_{\x\in\cX} \lp   s\sqrt{n}      \lp\sum_{k=2}^{r+1}c_k\h^{(k)}\rp^T\x  \rp \nonumber \\
  D^{(sph)}(s) & \triangleq  &   s \max_{\y\in\cY}
\lp \sqrt{n}  \y^T\lp\sum_{k=2}^{r+1}b_k\u^{(2,k)}\rp \rp.
 \end{eqnarray}
One then trivially has
\begin{eqnarray}\label{eq:prac4}
D^{(bin)}(s) & = & \max_{\x\in\cX}   \lp s\sqrt{n}      \lp\sum_{k=2}^{r+1}c_k\h^{(k)}\rp^T\x \rp =
       \sum_{i=1}^n \left |s\lp\sum_{k=2}^{r+1}c_k\h_i^{(k)}\rp \right |
=     \sum_{i=1}^n \left |\lp\sum_{k=2}^{r+1}c_k\h_i^{(k)}\rp \right |
=  \sum_{i=1}^n D^{(bin)}_i, \nonumber \\
 \end{eqnarray}
with
\begin{eqnarray}\label{eq:prac5}
D^{(bin)}_i(c_k)=\left |\lp\sum_{k=2}^{r+1}c_k\h_i^{(k)}\rp \right |.
\end{eqnarray}
Consequently,
  \begin{eqnarray}\label{eq:prac6}
\varphi(D^{(bin)}(s),\c) & = &
n \mE_{G,{\mathcal U}_{r+1}} \frac{1}{\c_r} \log
\lp \mE_{{\mathcal U}_{r}} \lp \dots \lp \mE_{{\mathcal U}_3}\lp\lp\mE_{{\mathcal U}_2} \lp
    e^{\c_2D_1^{(bin)}}  \rp\rp^{\frac{\c_3}{\c_2}}\rp\rp^{\frac{\c_4}{\c_3}} \dots \rp^{\frac{\c_{r}}{\c_{r-1}}}\rp
    = n\varphi(D_1^{(bin)}). \nonumber \\
   \end{eqnarray}
On the other hand, we also have
\begin{eqnarray}\label{eq:prac7}
   D^{(sph)}(s) & \triangleq  &   s \max_{\y\in\cY}
\lp \sqrt{n}  \y^T\lp\sum_{k=2}^{r+1}b_k\u^{(2,k)}\rp \rp
=  s\sqrt{n}  \left \|\sum_{k=2}^{r+1}b_k\u^{(2,k)}\right \|_2.
 \end{eqnarray}
We now utilize the following \emph{square root trick} introduced on numerous occasions in \cite{StojnicMoreSophHopBnds10,StojnicLiftStrSec13,StojnicGardSphErr13,StojnicGardSphNeg13}
\begin{eqnarray}\label{eq:prac8}
   D^{(sph)} (s)
& =  &  s\sqrt{n}  \left \|\sum_{k=2}^{r+1}b_k\u^{(2,k)}\right \|_2
=  s\sqrt{n}  \min_{\gamma} \lp \frac{\left \|\sum_{k=2}^{r+1}b_k\u^{(2,k)}\right \|_2^2}{4\gamma}+\gamma \rp \nonumber \\
 & =  & s\sqrt{n}  \min_{\gamma} \lp \frac{\sum_{i=1}^{m}\lp \sum_{k=2}^{r+1}b_k\u_i^{(2,k)}\rp^2}{4\gamma}+\gamma \rp.
 \end{eqnarray}
After scaling $\gamma=\gamma_{sq}\sqrt{n}$, (\ref{eq:prac8}) becomes
\begin{eqnarray}\label{eq:prac9}
   D^{(sph)}(s)
  & =  & s\sqrt{n}  \min_{\gamma_{sq}} \lp \frac{\sum_{i=1}^{m}\lp \sum_{k=2}^{r+1}b_k\u_i^{(2,k)}\rp^2}{4\gamma_{sq}\sqrt{n}}+\gamma_{sq}\sqrt{n} \rp = s \min_{\gamma_{sq}} \lp \frac{\sum_{i=1}^{m}\lp \sum_{k=2}^{r+1}b_k\u_i^{(2,k)}\rp^2}{4\gamma_{sq}}+\gamma_{sq}n \rp \nonumber \\
  & =  &  s \min_{\gamma_{sq}} \lp \sum_{i=1}^{m} D_i^{(sph)}+\gamma_{sq}n \rp, \nonumber \\
 \end{eqnarray}
with
\begin{eqnarray}\label{eq:prac10}
   D_i^{(sph)}(b_k)= \frac{\lp \sum_{k=2}^{r+1}b_k\u_i^{(2,k)}  \rp^2}{4\gamma_{sq}}.
 \end{eqnarray}
Below, we distinguish two types of Hopfield models: (i) the positive ones (obtained for $s=1$); and (ii) the negative ones (obtained for $s=-1$).

\subsection{Positive square root Hopfield models ($s=1$)}
\label{sec:pos}

Connecting the ground state energy of the positive square root Hopfield models, $f^+_{sq}$ given in (\ref{eq:limlogpartfunsqrt}), and the random primal, $\psi_{rp}(\cdot)$ given in Corollary \ref{cor:cor1}, we have
 \begin{eqnarray}
f_{sq}^{+}(\infty)
& = &  \lim_{n\rightarrow\infty}\frac{\mE_G \max_{\x\in\cX} \sqrt{\x^TG^TG\x}}{\sqrt{n}} =
 \lim_{n\rightarrow\infty}\frac{\mE_G \max_{\x\in\cX}    \max_{\y\in\cY} \y^TG\x}{\sqrt{n}} \nonumber \\
& =  &
    -\lim_{n\rightarrow\infty} \frac{\mE_G  \psi_{rp}}{\sqrt{n}}
   =
 -\lim_{n\rightarrow\infty} \psi_{rd}(\hat{\p},\hat{\q},\hat{\c},1,1,1),
  \label{eq:prac11}
\end{eqnarray}
where the non-fixed parts of $\hat{\p}$, $\hat{\q}$, and  $\hat{\c}$ are the solutions of the following system
\begin{eqnarray}\label{eq:prac12}
   \frac{d \psi_{rd}(\p,\q,\c,1,1,1)}{d\p}  =  0,\quad
   \frac{d \psi_{rd}(\p,\q,\c,1,1,1)}{d\q} =  0,\quad
   \frac{d \psi_{rd}(\p,\q,\c,1,1,1)}{d\c} =  0.
 \end{eqnarray}
Moreover, relying on (\ref{eq:prac1})-(\ref{eq:prac10}), we have
 \begin{eqnarray}
 \lim_{n\rightarrow\infty} \psi_{rd}(\hat{\p},\hat{\q},\hat{\c},1,1,1) =  \bar{\psi}_{rd}(\hat{\p},\hat{\q},\hat{\c},\hat{\gamma}_{sq},1,1,1),
  \label{eq:prac12a}
\end{eqnarray}
where
\begin{eqnarray}\label{eq:prac13}
    \bar{\psi}_{rd}(\p,\q,\c,\gamma_{sq},1,1,1)   & = &  \frac{1}{2}    \sum_{k=2}^{r+1}\Bigg(\Bigg.
   \p_{k-1}\q_{k-1}
   -\p_{k}\q_{k}
  \Bigg.\Bigg)
\c_k
\nonumber \\
& &  - \varphi(D_1^{(bin)}(c_k(\p,\q)),\c) -\gamma_{sq}- \alpha\varphi(D_1^{(sph)}(b_k(\p,\q)),\c),
  \end{eqnarray}
with $\varphi(D_1^{(bin)}(c_k(\p,\q)),\c)$ (based on (\ref{eq:prac2}), (\ref{eq:prac5}), and (\ref{eq:prac6})) given by
{\small   \begin{align}\label{eq:prac14}
\varphi(D_1^{(bin)}(c_k(\p,\q)),\c) & =
 \mE_{{\mathcal U}_{r+1}} \frac{1}{\c_r} \log
\lp \mE_{{\mathcal U}_{r}} \lp \dots \lp \mE_{{\mathcal U}_3}\lp\lp\mE_{{\mathcal U}_2} \lp
    e^{\c_2 \left | \sum_{k=2}^{r+1}c_k(\p,\q)\h_i^{(k)} \right |}  \rp\rp^{\frac{\c_3}{\c_2}}\rp\rp^{\frac{\c_4}{\c_3}} \dots \rp^{\frac{\c_{r}}{\c_{r-1}}}\rp, \nonumber \\
  \end{align}}

\noindent $\varphi(D_1^{(sph)}(b_k(\p,\q)),\c)$ (based on (\ref{eq:prac2}), (\ref{eq:prac9}), and (\ref{eq:prac10})) given by
 {\small \begin{align}\label{eq:prac15}
\varphi(D_1^{(sph)}(b_k(\p,\q)),\c) & =
 \mE_{{\mathcal U}_{r+1}} \frac{1}{\c_r} \log
\lp \mE_{{\mathcal U}_{r}} \lp \dots \lp \mE_{{\mathcal U}_3}\lp\lp\mE_{{\mathcal U}_2} \lp
    e^{\c_2 \frac{\lp \sum_{k=2}^{r+1}b_k(\p,\q)\u_i^{(2,k)}  \rp^2}{4\gamma_{sq}}}  \rp\rp^{\frac{\c_3}{\c_2}}\rp\rp^{\frac{\c_4}{\c_3}} \dots \rp^{\frac{\c_{r}}{\c_{r-1}}}\rp, \nonumber \\
\end{align}}

\noindent and with $\hat{\gamma}_{sq}$ and  the non-fixed parts of $\hat{\p}$, $\hat{\q}$, and  $\hat{\c}$ being the solutions of the following system
\begin{eqnarray}\label{eq:prac16}
   \frac{d \bar{\psi}_{rd}(\p,\q,\c,\gamma_{sq},1,1,1)}{d\p} & = & 0 \nonumber \\
   \frac{d \bar{\psi}_{rd}(\p,\q,\c,\gamma_{sq},1,1,1)}{d\q} & = & 0 \nonumber \\
   \frac{d \bar{\psi}_{rd}(\p,\q,\c,\gamma_{sq},1,1,1)}{d\c} & = & 0 \nonumber \\
   \frac{d \bar{\psi}_{rd}(\p,\q,\c,\gamma_{sq},1,1,1)}{d\gamma_{sq}} & = & 0,
 \end{eqnarray}
 and consequently
\begin{eqnarray}\label{eq:prac17}
c_k(\hat{\p},\hat{\q})  & = & \sqrt{\hat{\q}_{k-1}-\hat{\q}_k} \nonumber \\
b_k(\hat{\p},\hat{\q})  & = & \sqrt{\hat{\p}_{k-1}-\hat{\p}_k}.
 \end{eqnarray}
Connecting  (\ref{eq:prac11}) and (\ref{eq:prac12a}), we then find
 \begin{eqnarray}
f_{sq}^{+}(\infty)
& = &  \lim_{n\rightarrow\infty}\frac{\mE_G \max_{\x\in\cX}\pm \sqrt{\x^TG^TG\x}}{\sqrt{n}} =
 \lim_{n\rightarrow\infty}\frac{\mE_G \max_{\x\in\cX}    \max_{\y\in\cY} \y^TG\x}{\sqrt{n}} \nonumber \\
    &  = &
 -\lim_{n\rightarrow\infty} \psi_{rd}(\hat{\p},\hat{\q},\hat{\c},1,1,1)
 = -  \bar{\psi}_{rd}(\hat{\p},\hat{\q},\hat{\c},\hat{\gamma}_{sq},1,1,1) \nonumber \\
 & = &   -\frac{1}{2}    \sum_{k=2}^{r+1}\Bigg(\Bigg.
   \hat{\p}_{k-1}\hat{\q}_{k-1}
   -\hat{\p}_{k}\hat{\q}_{k}
  \Bigg.\Bigg)
\hat{\c}_k
  + \varphi(D_1^{(bin)}(c_k(\hat{\p},\hat{\q})),\c) + \hat{\gamma}_{sq} + \alpha\varphi(D_1^{(sph)}(b_k(\hat{\p},\hat{\q})),\c). \nonumber \\
  \label{eq:prac18}
\end{eqnarray}
The following theorem summarizes the above results.

\begin{theorem}
  \label{thme:thmprac1}
  Assume the complete sfl RDT setup of \cite{Stojnicsflgscompyx23}. Consider large $n$ linear regime with $\alpha=\lim_{n\rightarrow\infty} \frac{m}{n}$ and let $\varphi(\cdot)$ and $\bar{\psi}(\cdot)$ be as given in (\ref{eq:prac2}) and (\ref{eq:prac13}), respectively. Also, let the ``fixed'' parts of $\hat{\p}$, $\hat{\q}$, and $\hat{\c}$ satisfy $\hat{\p}_1\rightarrow 1$, $\hat{\q}_1\rightarrow 1$, $\hat{\c}_1\rightarrow 1$, $\hat{\p}_{r+1}=\hat{\q}_{r+1}=\hat{\c}_{r+1}=0$, and let the ``non-fixed'' parts of $\hat{\p}_k$, $\hat{\q}_k$, and $\hat{\c}_k$ ($k\in\{2,3,\dots,r\}$) be the solutions of (\ref{eq:prac16}). Moreover, let $c_k(\hat{\p},\hat{\q})$ and $b_k(\hat{\p},\hat{\q})$ be as in (\ref{eq:prac17}). Then
 \begin{eqnarray}
f_{sq}^{+}(\infty)
& = &     -\frac{1}{2}    \sum_{k=2}^{r+1}\Bigg(\Bigg.
   \hat{\p}_{k-1}\hat{\q}_{k-1}
   -\hat{\p}_{k}\hat{\q}_{k}
  \Bigg.\Bigg)
\hat{\c}_k
  + \varphi(D_1^{(bin)}(c_k(\hat{\p},\hat{\q})),\hat{\c}) + \hat{\gamma}_{sq} + \alpha\varphi(D_1^{(sph)}(b_k(\hat{\p},\hat{\q})),\hat{\c}). \nonumber \\
  \label{eq:thmprac1eq1}
\end{eqnarray}
\end{theorem}
\begin{proof}
Follows from the previous discussion, Theorem \ref{thm:thmsflrdt1}, Corollary \ref{cor:cor1}, and the sfl RDT machinery presented in \cite{Stojnicnflgscompyx23,Stojnicsflgscompyx23,Stojnicflrdt23}.
\end{proof}

\subsubsection{Numerical evaluations}
\label{sec:nuemrical}

Theorem \ref{thme:thmprac1} is sufficient to conduct numerical evaluations and obtain concrete values for the studied free energy. There are several analytical results that can be helpful in that regard and that we present below. As it will be soon clear, they mostly relate to the evaluation of the so-called \emph{spherical} part $\varphi(D_i^{(sph)}(c_k(\p,\q)),\c)$, where one can obtain convenient closed form expressions. For the easiness of the exposition, we start with the smallest possible $r$ and proceed by incrementally increasing it by $1$. Such a path eventually enables seeing in a neat and systematic way how the whole lifting mechanism  \emph{astonishingly quickly} converges. For the numerical values concreteness, we, on occasion, specialize the considerations to the most famous, so-called, square case, i.e., to the case $\alpha=1$.

\underline{1) \textbf{\emph{$r=1$ -- first level of lifting:}}} For $r=1$ one has that $\hat{\p}_1\rightarrow 1$ and $\hat{\q}_1\rightarrow 1$ which together with $\hat{\p}_{r+1}=\hat{\p}_{2}=\hat{\q}_{r+1}=\hat{\q}_{2}=0$, and $\hat{\c}_{2}\rightarrow 0$ gives
\begin{align}\label{eq:prac19}
    -\bar{\psi}_{rd}(\hat{\p},\hat{\q},\hat{\c},\gamma_{sq},1,1,1)   & =   -\frac{1}{2}
\c_2
  + \frac{1}{\c_2}\log\lp \mE_{{\mathcal U}_2} e^{\c_2|\sqrt{1-0}\h_1^{(2)} |}\rp +\gamma_{sq}
+ \alpha\frac{1}{\c_2}\log\lp \mE_{{\mathcal U}_2} e^{\c_2\frac{(\sqrt{1-0}\u_1^{(2,2)})^2}{4\gamma_{sq}}}\rp \nonumber \\
& \rightarrow
   \frac{1}{\c_2}\log\lp 1+ \mE_{{\mathcal U}_2} \c_2|\sqrt{1-0}\h_1^{(2)} |\rp +\gamma_{sq}
+ \alpha\frac{1}{\c_2}\log\lp 1+ \mE_{{\mathcal U}_2} \c_2\frac{(\sqrt{1-0}\u_1^{(2,2)})^2}{4\gamma_{sq}}\rp \nonumber \\
& \rightarrow
   \frac{1}{\c_2}\log\lp 1+ \c_2\sqrt{\frac{2}{\pi}}\rp +\gamma_{sq}
+ \alpha\frac{1}{\c_2}\log\lp 1+ \frac{\c_2}{4\gamma_{sq}}\rp \nonumber \\
& \rightarrow
  \sqrt{\frac{2}{\pi}}+\gamma_{sq}
+  \frac{\alpha}{4\gamma_{sq}}.
  \end{align}
One then easily finds $\hat{\gamma}_{sq}=\frac{\sqrt{\alpha}}{2}$ and
\begin{align}\label{eq:prac20}
  f^{+,1}_{sq}(\infty)=- \bar{\psi}_{rd}(\hat{\p},\hat{\q},\hat{\c},\hat{\gamma}_{sq},1,1,1)   & =
   \sqrt{\frac{2}{\pi}}+\sqrt{\alpha},
  \end{align}
which for $\alpha=1$ becomes
\begin{align}\label{eq:prac21}
  f^{+,1}_{sq}(\infty) =
   \sqrt{\frac{2}{\pi}}+1\rightarrow \bl{\mathbf{1.7979}}.
  \end{align}

\underline{2) \textbf{\emph{$r=2$ -- second level of lifting:}}} To make the presentation smoother, we split the second level of lifting into two subparts: (i) \emph{partial} second level of lifting; and (ii) \emph{full} second level of lifting.

For the \emph{partial} part, we have, when $r=2$, that $\hat{\p}_1\rightarrow 1$ and $\hat{\q}_1\rightarrow 1$, $\hat{\p}_{2}=\hat{\q}_{2}=0$, and $\hat{\p}_{r+1}=\hat{\p}_{3}=\hat{\q}_{r+1}=\hat{\q}_{3}=0$ but in general  $\hat{\c}_{2}\neq 0$. As above, one again has
\begin{align}\label{eq:prac22}
   - \bar{\psi}_{rd}(\hat{\p},\hat{\q},\c,\gamma_{sq},1,1,1)   & =  - \frac{1}{2}
\c_2
  + \frac{1}{\c_2}\log\lp \mE_{{\mathcal U}_2} e^{\c_2|\sqrt{1-0}\h_1^{(2)} |}\rp + \gamma_{sq}
+ \alpha\frac{1}{\c_2}\log\lp \mE_{{\mathcal U}_2} e^{\c_2\frac{(\sqrt{1-0}\u_1^{(2,2)})^2}{4\gamma_{sq}}}\rp.
   \end{align}
After taking the derivatives with respect to $\c_2$ and $\gamma_{sq}$, one, for $\alpha=1$, finds $\hat{\gamma}_{sq}=0.6173$ and $\hat{\c}_2=0.4246$, which gives
\begin{align}\label{eq:prac23}
\hspace{-2in}(\mbox{\emph{partial} second level:}) \qquad \qquad   f^{+,2}_{sq}(\infty) \rightarrow \bl{\mathbf{1.7832}}.
  \end{align}

For the \emph{full} part, we have the same setup as above with the exception that now, in general, (in addition to $\hat{\c}_{2}\neq 0$) one has $\p_2\neq0$ and $\q_2\neq0$. As above, we write
\begin{align}\label{eq:prac24}
   - \bar{\psi}_{rd}(\p,\q,\c,\gamma_{sq},1,1,1)   & =  - \frac{1}{2}
(1-\p_2\q_2)\c_2
  + \frac{1}{\c_2}\mE_{{\mathcal U}_3}\log\lp \mE_{{\mathcal U}_2} e^{\c_2|\sqrt{1-\q_2}\h_1^{(2)} +\sqrt{\q_2}\h_1^{(3)} |}\rp \nonumber \\
& \quad    + \gamma_{sq}
+ \alpha\frac{1}{\c_2}\mE_{{\mathcal U}_3} \log\lp \mE_{{\mathcal U}_2} e^{\c_2\frac{(\sqrt{1-\p_2}\u_i^{(2,2)}+\sqrt{\p_2}\u_i^{(2,3)})^2}{4\gamma_{sq}}}\rp \nonumber \\
& =  - \frac{1}{2}
(1-\p_2\q_2)\c_2
  + \frac{1}{\c_2}\mE_{{\mathcal U}_3}\log\lp \mE_{{\mathcal U}_2} e^{\c_2|\sqrt{1-\q_2}\h_1^{(2)} +\sqrt{\q_2}\h_1^{(3)} |}\rp \nonumber \\& \quad    + \gamma_{sq}
+ \alpha \lp -\frac{1}{2\c_2} \log \lp \frac{2\gamma_{sq}-\c_2(1-\p_2)}{2\gamma_{sq}} \rp  +  \frac{\p_2}{2(2\gamma_{sq}-\c_2(1-\p_2))}   \rp
   \end{align}
After taking the derivatives one, for $\alpha=1$, finds $\hat{\gamma}_{sq}=0.6941$, $\p_2=0.5315$, $\q_2=0.6320$, and $\hat{\c}_2=1.0056$, which gives
\begin{align}\label{eq:prac25}
\hspace{-2in}(\mbox{\emph{full} second level:}) \qquad \qquad  f^{+,2}_{sq}(\infty) \rightarrow \bl{\mathbf{1.7801}}.
  \end{align}

\underline{3) \textbf{\emph{$r=3$ -- third level of lifting:}}} Since the above already showed the idea behind the partial lifting, we now  immediately  consider \emph{full} third level of lifting and at the end discuss the specialization to the partial part. For $r=3$, we have that $\hat{\p}_1\rightarrow 1$ and $\hat{\q}_1\rightarrow 1$  as well as  $\hat{\p}_{r+1}=\hat{\p}_{4}=\hat{\q}_{r+1}=\hat{\q}_{4}=0$. As above, one again has
 \begin{align}\label{eq:prac26}
   - \bar{\psi}_{rd}(\p,\q,\c,\gamma_{sq},1,1,1)   & =  - \frac{1}{2}
(1-\p_2\q_2)\c_2 - \frac{1}{2}
(\p_2\q_2-\p_3\q_3)\c_3 \nonumber \\
&\quad  + \frac{1}{\c_3}\mE_{{\mathcal U}_4}\log\lp \mE_{{\mathcal U}_3} \lp  \mE_{{\mathcal U}_2} e^{\c_2|\sqrt{1-\q_2}\h_1^{(2)} +\sqrt{\q_2-\q_3}\h_1^{(3)}+\sqrt{\q_3}\h_1^{(4)}|} \rp^{\frac{\c_3}{\c_2}} \rp \nonumber \\
& \quad    + \gamma_{sq}
+ \alpha\frac{1}{\c_3}\mE_{{\mathcal U}_4} \log\lp \mE_{{\mathcal U}_4} \lp \mE_{{\mathcal U}_2} e^{\c_2\frac{(\sqrt{1-\p_2}\u_i^{(2,2)}+\sqrt{\p_2-\p_3}\u_i^{(2,3)}+\sqrt{\p_3}\u_i^{(2,4)})^2}{4\gamma_{sq}}}\rp^{\frac{\c_3}{\c_2}} \rp \nonumber \\
   & =  - \frac{1}{2}
(1-\p_2\q_2)\c_2 - \frac{1}{2}
(\p_2\q_2-\p_3\q_3)\c_3 \nonumber \\
&\quad  + \frac{1}{\c_3}\mE_{{\mathcal U}_4}\log\lp \mE_{{\mathcal U}_3} \lp  \mE_{{\mathcal U}_2} e^{\c_2|\sqrt{1-\q_2}\h_1^{(2)} +\sqrt{\q_2-\q_3}\h_1^{(3)}+\sqrt{\q_3}\h_1^{(4)}|} \rp^{\frac{\c_3}{\c_2}} \rp \nonumber \\
& \quad    + \gamma_{sq}
+ \alpha
\Bigg(\Bigg. -\frac{1}{2\c_2} \log \lp \frac{2\gamma_{sq}-\c_2(1-\p_2)}{2\gamma_{sq}} \rp \nonumber \\
& \quad -\frac{1}{2\c_3} \log \lp \frac{2\gamma_{sq}-\c_2(1-\p_2)-\c_3(\p_2-\p_3)}{2\gamma_{sq}-\c_2(1-\p_2)} \rp  \nonumber \\
& \quad +  \frac{\p_3}{2(2\gamma_{sq}-\c_2(1-\p_2)-\c_3(\p_2-\p_3))}   \Bigg.\Bigg)
   \end{align}
Taking the derivatives while fixing $\p_3=\q_3=0$ one obtains the \emph{partial} third level lifting and in particular for $\alpha=1$
\begin{align}\label{eq:prac27}
\hspace{-2in}(\mbox{\emph{partial} third level:}) \qquad \qquad f^{+,3}_{sq}(\infty) \rightarrow \bl{\mathbf{1.7791}}.
  \end{align}
On the other hand, taking the derivatives without fixing $\p_3=\q_3=0$ one obtains the \emph{full} third level lifting and in particular for $\alpha=1$
\begin{align}\label{eq:prac28}
\hspace{-2in}(\mbox{\emph{full} third level:}) \qquad \qquad   f^{+,3}_{sq}(\infty) \rightarrow \bl{\mathbf{1.7788}}.
  \end{align}
All relevant quantities for the third level of lifting (for both partial and full scenarios), together with the first and second level corresponding ones are systematically shown in Table \ref{tab:2rsbunifiedsqrtpos}.

\begin{table}[h]
\caption{$r$-sfl RDT parameters; \underline{positive} \emph{square-root} Hopfield model; $\alpha=1$; $\hat{\c}_1\rightarrow 1$; $n,\beta\rightarrow\infty$}\vspace{.1in}
\centering
\def\arraystretch{1.2}
\begin{tabular}{||l||c||c|c|c||c|c|c||c|c||c||}\hline\hline
 \hspace{-0in}$r$-sfl RDT                                             & $\hat{\gamma}_{sq}$    & $\hat{\p}_3$ & $\hat{\p}_2$ & $\hat{\p}_1`$        & $\hat{\q}_3$ & $\hat{\q}_2$  & $\hat{\q}_1$ &  $\hat{\c}_3$   & $\hat{\c}_2$    & $f_{sq}^{+,r}$  \\ \hline\hline
$\mathbf{1}$ (full)                                      & $0.5$ & $0$ & $0$  & $\rightarrow 1$  & $0$ & $0$ & $\rightarrow 1$ &  $\rightarrow 0$
 &  $\rightarrow 0$  & \bl{$\mathbf{1.7979}$} \\ \hline\hline
 $\mathbf{2}$ (partial)                                      & $0.6173$ & $0$ & $0$ & $\rightarrow 1$ & $0$ & $0$ & $\rightarrow 1$ &  $\rightarrow 0$
 &  $0.4246$   & \bl{$\mathbf{1.7832}$} \\ \hline
  $\mathbf{2}$ (full)                                      & $0.6941$ & $0$ & $0.5315$ & $\rightarrow 1$ & $0$ & $0.6320$ & $\rightarrow 1$ &  $\rightarrow 0$
 &  $1.0056$   & \bl{$\mathbf{1.7801}$}  \\ \hline\hline
 $\mathbf{3}$ (partial)                                      & $0.7434$ & $0$ & $0.7510$  & $\rightarrow 1$ & $0$ & $0.8397$ & $\rightarrow 1$ &  $0.2508$
 &  $1.7762$   & \bl{$\mathbf{1.7791}$}  \\ \hline
  $\mathbf{3}$ (full)                                      & $0.7752$ & $0.36$ & $0.8504$  & $\rightarrow 1$ & $0.4427$ & $0.9160$ & $\rightarrow 1$ &  $0.5044$
 &  $2.6825$   & \bl{$\mathbf{1.7788}$}  \\ \hline\hline
\end{tabular}
\label{tab:2rsbunifiedsqrtpos}
\end{table}

\underline{4) \textbf{General \emph{$r$-th level of lifting:}}}  For the completeness, we state the form of $\bar{\psi}_{rd}(\p,\q,\c,\gamma_{sq},1,1,1) $ obtained for general $r$. Following the above procedures, we basically have
 \begin{align}\label{eq:prac29}
   - \bar{\psi}_{rd}(\p,\q,\c,\gamma_{sq},1,1,1)   & =  - \frac{1}{2}
\sum_{k=2}^{r+1}(\p_{k-1}\q_{k-1}-\p_k\q_k)\c_k   \nonumber \\
&\quad  + \frac{1}{\c_r}\mE_{{\mathcal U}_{r+1}}\log\lp \mE_{{\mathcal U}_r} \lp \dots \lp \mE_{{\mathcal U}_3} \lp  \mE_{{\mathcal U}_2} e^{\c_2|\sum_{k=2}^{r+1}\sqrt{\q_{k-1}-\q_k}\h_1^{(k)} |} \rp^{\frac{\c_3}{\c_2}} \rp^{\frac{\c_4}{\c_{3}}} \dots \rp^{\frac{\c_r}{\c_{r-1}}} \rp
\nonumber \\
 & \quad    + \gamma_{sq}
+ \alpha
\Bigg(\Bigg. -\sum_{k=2}^{r}\frac{1}{2\c_k} \log \lp \frac{\Theta_{k}}{\Theta_{k-1}}\rp +\frac{\p_r}{2\Theta_{r}}  \Bigg.\Bigg),
    \end{align}
where
\begin{eqnarray}\label{eq:prac30}
 \Theta_{1} & = & 2\gamma_{sq} \nonumber \\
 \Theta_{k} & = & \Theta_{k-1}-\c_k(\p_{k-1}-\p_k), k\in\{2,3,\dots,r\}.
 \end{eqnarray}
The above characterization of $\bar{\psi}(\cdot)$ gives explicit dependence on all key parameters $\p$, $\q$, $\m$, and $\gamma_{sq}$ and consequently  enables direct computations of all the needed derivatives for any $r\in\mN$. Several additional helpful relations can be deduced from \cite{Stojnicsflgscompyx23}. We show them nest.

\subsubsection{Helpful relations}
\label{sec:helprel}

We first need a few definitions and to that end start with
\begin{eqnarray}
\label{eq:thm3eq2}
Z_{(s)} & \triangleq & \sum_{i_1=1}^{l}\lp\sum_{i_2=1}^{l}e^{\beta D_0^{(i_1,i_2)}} \rp^{s}
=\sum_{i_1=1}^{l}\lp\sum_{i_2=1}^{l}  A^{(i_1,i_2)} \rp^{s}
=\sum_{i_1=1}^{l}\lp  C^{(i_1)} \rp^{s}
\nonumber \\
 D_0^{(i_1,i_2)} & \triangleq &   (\y^{(i_2)})^T\lp\sum_{k=1}^{r+1}b_k\u^{(2,k)}\rp +\lp\sum_{k=1}^{r+1}c_k\h^{(k)}\rp^T\x^{(i_1)}
 \end{eqnarray}
With $\m_0=1$, we set
\begin{eqnarray}\label{eq:rthlev2genanal7a}
\zeta_r\triangleq \mE_{{\mathcal U}_{r}} \lp \dots \lp \mE_{{\mathcal U}_2}\lp\lp\mE_{{\mathcal U}_1} \lp Z^{\frac{\m_1}{\m_0}}\rp\rp^{\frac{\m_2}{\m_1}}\rp\rp^{\frac{\m_3}{\m_2}} \dots \rp^{\frac{\m_{r}}{\m_{r-1}}}, r\geq 1.
\end{eqnarray}
One can then write
\begin{eqnarray}\label{eq:rthlev2genanal7b}
\zeta_k = \mE_{{\mathcal U}_{k}} \lp  \zeta_{k-1} \rp^{\frac{\m_{k}}{\m_{k-1}}}, k\geq 2,\quad \mbox{and} \quad
\zeta_0=Z, \qquad \zeta_1=\mE_{{\mathcal U}_1} \lp Z^{\frac{\m_1}{\m_0}}\rp.
\end{eqnarray}
 Also, consider the operators
\begin{eqnarray}\label{eq:thm3eq3}
 \Phi_{{\mathcal U}_k} & \triangleq &  \mE_{{\mathcal U}_{k}} \frac{\zeta_{k-1}^{\frac{\m_k}{\m_{k-1}}}}{\zeta_k},
 \end{eqnarray}
and set
\begin{eqnarray}\label{eq:thm3eq4}
  \gamma_0(i_1,i_2) & = &
\frac{(C^{(i_1)})^{s}}{Z_{(s)}}  \frac{A^{(i_1,i_2)}}{C^{(i_1)}} \nonumber \\
\gamma_{01}^{(r)}  & = & \prod_{k=r}^{1}\Phi_{{\mathcal U}_k} (\gamma_0(i_1,i_2)) \nonumber \\
\gamma_{02}^{(r)}  & = & \prod_{k=r}^{1}\Phi_{{\mathcal U}_k} (\gamma_0(i_1,i_2)\times \gamma_0(i_1,p_2)) \nonumber \\
\gamma_{k_1+1}^{(r)}  & = & \prod_{k=r}^{k_1+1}\Phi_{{\mathcal U}_k} \lp \prod_{k=k_1}^{1}\Phi_{{\mathcal U}_k}\gamma_0(i_1,i_2)\times \prod_{k=k_1}^{1} \Phi_{{\mathcal U}_k}\gamma_0(p_1,p_2) \rp.
 \end{eqnarray}
Then,\cite{Stojnicsflgscompyx23} gives
\begin{eqnarray}\label{eq:thm3eq5}
\hat{\p}_{k-1} & = & E_{{\mathcal U}_{r+1}}\langle (\x^{(i_1)})^T\x^{(p_1)}\rangle_{\gamma_{k}^{(r)}} \nonumber \\
\hat{\q}_{k-1} & = & E_{{\mathcal U}_{r+1}}\langle (\y^{(i_2)})^T\y^{(p_2)}\rangle_{\gamma_{k}^{(r)}},
 \end{eqnarray}
 where $\langle  \cdot \rangle_{\gamma_{k}^{(r)}}$ stands for the average with respect to $\gamma_{k}^{(r)}$.
For example, to make everything concrete, one can take $r=3$ and for $\beta=\bar{\beta}\sqrt{n}$, $\bar{\beta}\rightarrow\infty$ and $b_k=\hat{b}_k$ and $c_k=\hat{c}_k$,  observe that
\begin{eqnarray}\label{eq:thm3eq5a1}
Z_{(s)} & \rightarrow & e^{\max_{\x} s \max_{\y}  \beta D_0^{(i_1,i_2)} }
\rightarrow e^{  \bar{\beta} \lp \sqrt{n}\|\sum_{k=1}^{r+1}\hat{b}_k\u^{(2,k)}\|_2 +\lp\sum_{k=1}^{r+1}\hat{c}_k\h^{(k)}\rp^T\bar{\x} \rp } \nonumber \\
&\rightarrow & e^{  \bar{\beta} \lp \frac{\|\sum_{k=1}^{r+1}\hat{b}_k\u^{(2,k)}\|_2^2}{4\hat{\gamma}_{sq}}+n\hat{\gamma}_{sq} +\lp\sum_{k=1}^{r+1}\hat{c}_k\h^{(k)}\rp^T\bar{\x} \rp } \nonumber \\
&\rightarrow & e^{  \bar{\beta} \lp \frac{\bar{\y}^T \lp \sum_{k=1}^{r+1}\hat{b}_k\u^{(2,k)}\rp }{2}+n\hat{\gamma}_{sq} +\lp\sum_{k=1}^{r+1}\hat{c}_k\h^{(k)}\rp^T\bar{\x} \rp } \nonumber \\
 \end{eqnarray}
where (since $\hat{\p}_1\rightarrow 1$ and $\hat{\q}_1\rightarrow 1$) the optimal $\bar{\x}_{i}$ and $\bar{\y}_{i}$ are given as
\begin{eqnarray}\label{eq:thm3eq6}
\bar{\x}_{i} & = & \mbox{sign}\lp\sqrt{1-\hat{\q}_2}\h_i^{(2)}+\sqrt{\hat{\q}_2-\hat{\q}_3}\h_i^{(3)}+\sqrt{\hat{\q}_3}\h_i^{(4)}\rp \nonumber \\
\bar{\y}_{i} & = &\frac{\lp\sqrt{1-\hat{\p}_2}\u_i^{(2,2)}+\sqrt{\hat{\p}_2-\hat{\p}_3}\u_i^{(2,3)}+\sqrt{\hat{\p}_3}\u_i^{(2,4)}\rp}{2\hat{\gamma}_{sq}}.
  \end{eqnarray}
Relying on (\ref{eq:prac5}), one further has
\begin{eqnarray}\label{eq:thm3eq6a0}
D^{(bin)}_1(\hat{c}_k)=\left |\lp\sum_{k=2}^{r+1}\hat{c}_k\h_i^{(k)}\rp \right |
=\left |\lp\sum_{k=2}^{r+1}\sqrt{\hat{\q}_{k-1}-\hat{\q}_k}\h_i^{(k)}\rp \right |
=\left | \sqrt{1-\q_2}\h_1^{(2)}+\sqrt{\q_2-\q_3}\h_1^{(3)}+\sqrt{\q_3}\h_1^{(4)} \right |,
\end{eqnarray}
 and, analogously, relying on (\ref{eq:prac10})
\begin{eqnarray}\label{eq:thm3eq6a01}
   D_1^{(sph)}(\hat{b}_k)= \frac{\lp \sum_{k=2}^{r+1}\hat{b}_k\u_i^{(2,k)}  \rp^2}{4\gamma_{sq}}
   =\frac{\lp \sqrt{1-\hat{\p}_2}\u_1^{(2,2)}+\sqrt{\hat{\p}_2-\hat{\p}_3}\u_1^{(2,3)}+\sqrt{\hat{\p}_3}\u_1^{(2,4)} \rp^2}{4\gamma_{sq}}.
 \end{eqnarray}
 We recall that triplets $\u^{(2,k)}$ and $\h^{(k)}$ are independent of each other and that each of these vectors is comprised of i.i.d. components, which allows a decoupling over $i$. Moreover, due to symmetry of indices $i$ we also have from (\ref{eq:thm3eq5})
\begin{eqnarray}\label{eq:thm3eq602}
\hat{\p}_{k-1} & = & E_{{\mathcal U}_{r+1}}\langle (\x^{(i_1)})^T\x^{(p_1)}\rangle_{\gamma_{k}^{(r)}}
= n E_{{\mathcal U}_{r+1}}\left \langle \frac{\bar{\x}_1^{(i_1)}}{\sqrt{n}} \frac{\bar{\x}_1^{(p_1)}}{\sqrt{n}}\right \rangle_{\gamma_{k}^{(r)}}
=  E_{{\mathcal U}_{r+1}}\left \langle \bar{\x}_1^{(i_1)}\bar{\x}_1^{(p_1)}\right \rangle_{\gamma_{k}^{(r)}}
\nonumber \\
\hat{\q}_{k-1} & = & E_{{\mathcal U}_{r+1}}\langle (\y^{(i_2)})^T\y^{(p_2)}\rangle_{\gamma_{k}^{(r)}}
= n E_{{\mathcal U}_{r+1}}\left \langle \frac{\bar{\y}_1^{(i_2)}}{\sqrt{n}} \frac{\bar{\y}_1^{(p_2)}}{\sqrt{n}}\right \rangle_{\gamma_{k}^{(r)}}
=  E_{{\mathcal U}_{r+1}}\left \langle \bar{\y}_1^{(i_2)}\bar{\y}_1^{(p_2)}\right \rangle_{\gamma_{k}^{(r)}}
 \end{eqnarray}
 where $\bar{\x}^{(i_1)}$ is optimal over $i_1$ indexing system and $\bar{\x}^{(p_1)}$ is optimal over $p_1$ indexing system and, analogously, $\bar{\y}^{(i_2)}$ is optimal over $i_2$ indexing system and $\bar{\y}^{(p_2)}$ is optimal over $p_2$ indexing system. From (\ref{eq:thm3eq602}),
we then have
\begin{equation}\label{eq:thm3eq7}
\hat{\p}_2=\mE_{{\mathcal U}_{4}}
\lp\frac{
\mE_{{\mathcal U}_{3}}
\lp
\lp \frac{\mE_{{\mathcal U}_{2}}\bar{\x}_1 e^{\hat{\c}_2D^{(bin)}_1(\hat{c}_k)}}
{\mE_{{\mathcal U}_{2}}e^{\hat{\c}_2 D^{(bin)}_1(\hat{c}_k)}}\rp^2
\lp\mE_{{\mathcal U}_{2}}e^{\hat{\c}_2D^{(bin)}_1(\hat{c}_k)}\rp^{\frac{\c_3}{\c_2}}\rp}
{\mE_{{\mathcal U}_{3}}\lp \lp \mE_{{\mathcal U}_{2}} e^{\hat{\c}_2D^{(bin)}_1(\hat{c}_k)}\rp^{\frac{\c_3}{\c_2}}\rp} \rp,
  \end{equation}
and
\begin{equation}\label{eq:thm3eq8}
\hat{\p}_3= \mE_{{\mathcal U}_{4}}
\lp\frac{
\mE_{{\mathcal U}_{3}}
\lp
\lp \frac{\mE_{{\mathcal U}_{2}}\bar{\x}_1 e^{\hat{\c}_2D^{(bin)}_1(\hat{c}_k)}}
{\mE_{{\mathcal U}_{2}}e^{\hat{\c}_2 D^{(bin)}_1(\hat{c}_k)}}\rp
\lp\mE_{{\mathcal U}_{2}}e^{\hat{\c}_2D^{(bin)}_1(\hat{c}_k)}\rp^{\frac{\c_3}{\c_2}}\rp}
{\mE_{{\mathcal U}_{3}}\lp \lp \mE_{{\mathcal U}_{2}} e^{\hat{\c}_2D^{(bin)}_1(\hat{c}_k)}\rp^{\frac{\c_3}{\c_2}}\rp} \rp^2.
  \end{equation}
Analogously to (\ref{eq:thm3eq7}) and (\ref{eq:thm3eq8}), we also have
\begin{equation}\label{eq:thm3eq7a1}
\hat{\q}_2=\alpha \mE_{{\mathcal U}_{4}}
\lp\frac{
\mE_{{\mathcal U}_{3}}
\lp
\lp \frac{\mE_{{\mathcal U}_{2}}\bar{\y}_1 e^{\hat{\c}_2 \lp D^{(sph)}_1(\hat{b}_k) +\hat{\gamma}_{sq}\rp}}
{\mE_{{\mathcal U}_{2}}e^{\hat{\c}_2 \lp D^{(sph)}_1(\hat{b}_k)+\hat{\gamma}_{sq}\rp}}\rp^2
\lp\mE_{{\mathcal U}_{2}}e^{\hat{\c}_2 \lp D^{(sph)}_1(\hat{b}_k)+\hat{\gamma}_{sq}\rp}\rp^{\frac{\c_3}{\c_2}}\rp}
{\mE_{{\mathcal U}_{3}}\lp \lp \mE_{{\mathcal U}_{2}} e^{\hat{\c}_2 \lp D^{(sph)}_1(\hat{b}_k)+\hat{\gamma}_{sq}\rp}\rp^{\frac{\c_3}{\c_2}}\rp} \rp,
  \end{equation}
and
\begin{equation}\label{eq:thm3eq8a1}
\hat{\q}_3=\alpha \mE_{{\mathcal U}_{4}}
\lp\frac{
\mE_{{\mathcal U}_{3}}
\lp
\lp \frac{\mE_{{\mathcal U}_{2}}\bar{\y}_1 e^{\hat{\c}_2 \lp D^{(sph)}_1(\hat{b}_k) +\hat{\gamma}_{sq}\rp}}
{\mE_{{\mathcal U}_{2}}e^{\hat{\c}_2 \lp D^{(sph)}_1(\hat{b}_k)+\hat{\gamma}_{sq}\rp}}\rp
\lp\mE_{{\mathcal U}_{2}}e^{\hat{\c}_2 \lp D^{(sph)}_1(\hat{b}_k)+\hat{\gamma}_{sq}\rp}\rp^{\frac{\c_3}{\c_2}}\rp}
{\mE_{{\mathcal U}_{3}}\lp \lp \mE_{{\mathcal U}_{2}} e^{\hat{\c}_2 \lp D^{(sph)}_1(\hat{b}_k)+\hat{\gamma}_{sq}\rp}\rp^{\frac{\c_3}{\c_2}}\rp} \rp^2,
  \end{equation}
It is now simple to write the above equations for general $r$. First, with $\c_1\rightarrow 1$,  we set
\begin{eqnarray}\label{eq:rthlev2genanal7ab1}
\zeta_r^{(bin)} &\triangleq & \mE_{{\mathcal U}_{r}} \lp \dots \lp \mE_{{\mathcal U}_3}\lp\lp\mE_{{\mathcal U}_2} \lp \lp e^{D^{(bin)}_1(\hat{c}_k)}\rp^{\frac{\c_2}{\c_1}}\rp\rp^{\frac{\c_3}{\c_2}}\rp\rp^{\frac{\c_4}{\c_3}} \dots \rp^{\frac{\c_{r}}{\c_{r-1}}}, r\geq 2\nonumber \\
\zeta_r^{(sph)} & \triangleq & \mE_{{\mathcal U}_{r}} \lp \dots \lp \mE_{{\mathcal U}_3}\lp\lp\mE_{{\mathcal U}_2} \lp \lp e^{\lp D^{(sph)}_1(\hat{b}_k) +\hat{\gamma}_{sq}\rp}\rp^{\frac{\c_2}{\c_1}}\rp\rp^{\frac{\c_3}{\c_2}}\rp\rp^{\frac{\c_4}{\c_3}} \dots \rp^{\frac{\c_{r}}{\c_{r-1}}}, r\geq 2\nonumber \\
 \end{eqnarray}
One can then write
\begin{align}\label{eq:rthlev2genanal7bb1}
\zeta_k^{(bin)} & =  \mE_{{\mathcal U}_{k}} \lp  \zeta_{k-1}^{(bin)} \rp^{\frac{\c_{k}}{\c_{k-1}}}, k\geq 3,\quad
\zeta_1^{(bin)}=e^{D^{(bin)}_1(\hat{c}_k)}, \hspace{.05in} \zeta_2^{(bin)}=\mE_{{\mathcal U}_2} \lp \lp e^{D^{(bin)}_1(\hat{c}_k)} \rp^{\frac{\c_2}{\c_1}}\rp \nonumber \\
\zeta_k^{(sph)} & =  \mE_{{\mathcal U}_{k}} \lp  \zeta_{k-1}^{(sph)} \rp^{\frac{\c_{k}}{\c_{k-1}}}, k\geq 3,\quad
\zeta_1^{(sph)}=e^{\lp D^{(sph)}_1(\hat{b}_k)+\hat{\gamma}_{sq}\rp},  \hspace{.05in}  \zeta_2^{(sph)}=\mE_{{\mathcal U}_2} \lp \lp e^{\lp D^{(sph)}_1(\hat{b}_k)+\hat{\gamma}_{sq}\rp } \rp^{\frac{\c_2}{\c_1}}\rp. \nonumber \\
\end{align}
 Also, for $k\geq 2$, let
\begin{eqnarray}\label{eq:thm3eq3b1}
 \Phi_{{\mathcal U}_k}^{(bin)}(\cdot) & \triangleq &  \mE_{{\mathcal U}_{k}} \frac{\lp \zeta_{k-1}^{(bin)} \rp^{\frac{\c_k}{\c_{k-1}}}}{\zeta_k^{(bin)}} \lp \cdot\rp \nonumber \\
 \Phi_{{\mathcal U}_k}^{(sph)} (\cdot)& \triangleq &  \mE_{{\mathcal U}_{k}} \frac{\lp \zeta_{k-1}^{(sph)} \rp^{\frac{\c_k}{\c_{k-1}}}}{\zeta_k^{(sph)}} \lp \cdot \rp.
 \end{eqnarray}
 Then
\begin{eqnarray}\label{eq:ovlpeq1}
 \hat{\p}_2& = & \prod_{k=r}^{3}\Phi_{{\mathcal U}_k}^{(bin)} \lp\lp\Phi_{{\mathcal U}_2}^{(bin)}(\bar{\x}_1) \rp^2 \rp \nonumber \\
 \hat{\p}_3 & = &\prod_{k=r}^{4}\Phi_{{\mathcal U}_k}^{(bin)} \lp \lp \Phi_{{\mathcal U}_3}^{(bin)} \lp\Phi_{{\mathcal U}_2}^{(bin)}(\bar{\x}_1) \rp\rp^2 \rp \nonumber \\
\vdots & = & \vdots \nonumber \\
 \hat{\p}_{k_1} & = &\prod_{k=r}^{k_1+1}\Phi_{{\mathcal U}_k}^{(bin)} \lp \lp \prod_{k=k_1}^{2}\Phi_{{\mathcal U}_{k}}^{(bin)}  (\bar{\x}_1) \rp^2 \rp \nonumber \\
\vdots & = & \vdots \nonumber \\
 \hat{\p}_{r} & = & \lp \prod_{k=r}^{2}\Phi_{{\mathcal U}_{k}}^{(bin)}  (\bar{\x}_1) \rp^2, \nonumber \\
 \end{eqnarray}
and
\begin{eqnarray}\label{eq:ovlpeq2}
 \hat{\q}_2& = & \alpha\prod_{k=r}^{3}\Phi_{{\mathcal U}_k}^{(bin)} \lp\lp\Phi_{{\mathcal U}_2}^{(bin)}(\bar{\y}_1) \rp^2 \rp \nonumber \\
 \hat{\q}_3 & = & \alpha\prod_{k=r}^{4}\Phi_{{\mathcal U}_k}^{(bin)} \lp \lp \Phi_{{\mathcal U}_3}^{(bin)} \lp\Phi_{{\mathcal U}_2}^{(bin)}(\bar{\y}_1) \rp\rp^2 \rp \nonumber \\
\vdots & = & \vdots \nonumber \\
 \hat{\q}_{k_1} & = & \alpha \prod_{k=r}^{k_1+1}\Phi_{{\mathcal U}_k}^{(bin)} \lp \lp \prod_{k=k_1}^{2}\Phi_{{\mathcal U}_{k}}^{(bin)}  (\bar{\y}_1) \rp^2 \rp \nonumber \\
\vdots & = & \vdots \nonumber \\
 \hat{\q}_{r} & = & \alpha \lp \prod_{k=r}^{2}\Phi_{{\mathcal U}_{k}}^{(bin)}  (\bar{\y}_1) \rp^2,
 \end{eqnarray}
 and the products are running in the \emph{index decreasing order}.

 In (\ref{eq:prac29}), the integrals containing $\hat{\q}$ are not easily solvable in closed form. Having that in mind, one observes that the above relations, (\ref{eq:thm3eq7}), (\ref{eq:thm3eq8}), and ultimately (\ref{eq:ovlpeq1}), offer an advantageous avenue for numerical evaluations. On the other hand, as is clear from (\ref{eq:prac29}), the integrals containing $\hat{\p}$ are already solved in closed form. It is therefore numerically more advantageous to directly differentiate (\ref{eq:prac29}) over $\p$ than to rely on (\ref{eq:thm3eq7a1}), (\ref{eq:thm3eq8a1}), and ultimately (\ref{eq:ovlpeq2}). Nonetheless, we included (\ref{eq:thm3eq7a1}), (\ref{eq:thm3eq8a1}), and (\ref{eq:ovlpeq2}) for the completeness. Finally, one of the key properties of the above mechanism that is of incredible practical relevance should not go unnoticed. Namely, in a rather remarkable fashion, the entire lifting mechanism exhibits a rapid convergence. In particular, looking at the values given in Table \ref{tab:2rsbunifiedsqrtpos}, one observes that the convergence of the ground state energy on the the fourth decimal is astonishingly achieved already on the third level of full lifting.

\subsubsection{Modulo-$\m$ sfl RDT}
\label{sec:posmodm}

It is interesting to note that the above can be repeated if one relies on the so-called modulo-$m$ sfl RDT frame of \cite{Stojnicsflgscompyx23}. Instead of Theorem \ref{thme:thmprac1}, we then basically have the following theorem.

\begin{theorem}
  \label{thme:thmprac2}
Assume the setup of Theorem \ref{thme:thmprac1} and instead of the complete, assume the modulo-$\m$ sfl RDT setup of \cite{Stojnicsflgscompyx23}. Let
 \begin{align}\label{eq:thmprac2eq1}
   - \bar{\psi}_{rd}(\p,\q,\c,\gamma_{sq},1,1,1)   & =  - \frac{1}{2}
\sum_{k=2}^{r+1}(\p_{k-1}\q_{k-1}-\p_k\q_k)\c_k   \nonumber \\
&\quad  + \frac{1}{\c_r}\mE_{{\mathcal U}_{r+1}}\log\lp \mE_{{\mathcal U}_r} \lp \dots \lp \mE_{{\mathcal U}_3} \lp  \mE_{{\mathcal U}_2} e^{\c_2|\sum_{k=2}^{r+1}\sqrt{\q_{k-1}-\q_k}\h_1^{(k)} |} \rp^{\frac{\c_3}{\c_2}} \rp^{\frac{\c_4}{\c_{3}}} \dots \rp^{\frac{\c_r}{\c_{r-1}}} \rp
\nonumber \\
 & \quad    + \gamma_{sq}
+ \alpha
\Bigg(\Bigg. -\sum_{k=2}^{r}\frac{1}{2\c_k} \log \lp \frac{\Theta_{k}}{\Theta_{k-1}}\rp +\frac{\p_r}{2\Theta_{r}}  \Bigg.\Bigg) \nonumber \\
  \Theta_{1} & =  2\gamma_{sq} \nonumber \\
 \Theta_{k} & =  \Theta_{k-1}-\c_k(\p_{k-1}-\p_k), k\in\{2,3,\dots,r\}.
 \end{align}
Let $\hat{\gamma}_{sq}$ and the ``non-fixed" parts of $\hat{\p}$ and $\hat{\q}$ be the solutions of the following system of equations
\begin{eqnarray}\label{eq:thmprac2eq2}
   \frac{d \bar{\psi}_{rd}(\p,\q,\c,\gamma_{sq},1,1,1)}{d\p} =  0 \nonumber \\
   \frac{d \bar{\psi}_{rd}(\p,\q,\c,\gamma_{sq},1,1,1)}{d\q} =  0 \nonumber \\
    \frac{d \bar{\psi}_{rd}(\p,\q,\c,\gamma_{sq},1,1,1)}{d\gamma_{sq}} =  0.
 \end{eqnarray}
Then
 \begin{eqnarray}
f_{sq,\m}^{+}(\infty)
& \geq  & \min_{\c} \lp - \bar{\psi}_{rd}(\hat{\p},\hat{\q},\c,\hat{\gamma}_{sq},1,1,1)\rp.
  \label{eq:thmprac2eq3}
\end{eqnarray}
\end{theorem}
\begin{proof}
Follows from the previous discussion, Theorem \ref{thm:thmsflrdt1}, Corollary \ref{cor:cor1}, and the sfl RDT machinery presented in \cite{Stojnicnflgscompyx23,Stojnicsflgscompyx23,Stojnicflrdt23}.
\end{proof}

We conducted the numerical evaluations using the modulo-$\m$ results of the above theorem for all $r$ from Table \ref{tab:2rsbunifiedsqrtpos}, i.e., for all $r\leq 3$. Exploring a very wide range of $\c$ parameters, we have found that the inequality in (\ref{eq:thmprac2eq3}) is in fact tight, i.e., that $f_{sq}^{+,r}(\infty)=f_{sq,\m}^{+,r}(\infty)$. This indicates that the \emph{stationarity} over $\c$ is actually of the \emph{minimization} type.

\subsection{Negative square root Hopfield model ($s=-1$)}
\label{sec:neg}

Following what was done above when considering the positive vaiant of the model, we start by connecting the ground state energy of the negative square root Hopfield model, $f^-_{sq}$ given in (\ref{eq:limlogpartfunsqrt}), and the random primal, $\psi_{rp}(\cdot)$, given in Corollary \ref{cor:cor1}. To that end, we have
 \begin{eqnarray}
-f_{sq}^{-}(\infty)
& = &  \lim_{n\rightarrow\infty}\frac{\mE_G \min_{\x\in\cX} \sqrt{\x^TG^TG\x}}{\sqrt{n}} =
-\lim_{n\rightarrow\infty}\frac{\mE_G \max_{\x\in\cX} - \sqrt{\x^TG^TG\x}}{\sqrt{n}}\nonumber \\
& = &
- \lim_{n\rightarrow\infty}\frac{\mE_G \max_{\x\in\cX}  -  \max_{\y\in\cY} \y^TG\x}{\sqrt{n}}
 =
    \lim_{n\rightarrow\infty} \frac{\mE_G  \psi_{rp}}{\sqrt{n}}
   =
 \lim_{n\rightarrow\infty} \psi_{rd}(\hat{\p},\hat{\q},\hat{\c},1,1,-1),
  \label{eq:negprac11}
\end{eqnarray}
where the non-fixed parts of $\hat{\p}$, $\hat{\q}$, and  $\hat{\c}$ are the solutions of the following system
\begin{eqnarray}\label{eq:negprac12}
   \frac{d \psi_{rd}(\p,\q,\c,1,1,-1)}{d\p} =  0,\quad
   \frac{d \psi_{rd}(\p,\q,\c,1,1,-1)}{d\q} =  0,\quad
   \frac{d \psi_{rd}(\p,\q,\c,1,1,-1)}{d\c} =  0.
 \end{eqnarray}
Moreover, relying on (\ref{eq:prac1})-(\ref{eq:prac10}), we have
 \begin{eqnarray}
 \lim_{n\rightarrow\infty} \psi_{rd}(\hat{\p},\hat{\q},\hat{\c},1,1,-1) =  \bar{\psi}_{rd}(\hat{\p},\hat{\q},\hat{\c},\hat{\gamma}_{sq},1,1,-1),
  \label{eq:negprac12a}
\end{eqnarray}
where
\begin{eqnarray}\label{eq:negprac13}
    \bar{\psi}_{rd}(\p,\q,\c,\gamma_{sq},1,1,-1)   & = &  \frac{1}{2}    \sum_{k=2}^{r+1}\Bigg(\Bigg.
   \p_{k-1}\q_{k-1}
   -\p_{k}\q_{k}
  \Bigg.\Bigg)
\c_k
\nonumber \\
& &  - \varphi(D_1^{(bin)}(c_k(\p,\q)),\c) +\gamma_{sq}- \alpha\varphi(-D_1^{(sph)}(b_k(\p,\q)),\c).
  \end{eqnarray}
 Connecting  (\ref{eq:negprac11}), (\ref{eq:negprac12a}), and (\ref{eq:negprac13}), we then find
 \begin{eqnarray}
-f_{sq}^{-}(\infty)
& = &  -\lim_{n\rightarrow\infty}\frac{\mE_G \max_{\x\in\cX} - \sqrt{\x^TG^TG\x}}{\sqrt{n}} =
 -\lim_{n\rightarrow\infty}\frac{\mE_G \max_{\x\in\cX}  -  \max_{\y\in\cY} \y^TG\x}{\sqrt{n}} \nonumber \\
    &  = &
 \lim_{n\rightarrow\infty} \psi_{rd}(\hat{\p},\hat{\q},\hat{\c},1,1,-1)
 =   \bar{\psi}_{rd}(\hat{\p},\hat{\q},\hat{\c},\hat{\gamma}_{sq},1,1,-1) \nonumber \\
 & = &   \frac{1}{2}    \sum_{k=2}^{r+1}\Bigg(\Bigg.
   \hat{\p}_{k-1}\hat{\q}_{k-1}
   -\hat{\p}_{k}\hat{\q}_{k}
  \Bigg.\Bigg)
\hat{\c}_k
  - \varphi(D_1^{(bin)}(c_k(\hat{\p},\hat{\q})),\c) + \hat{\gamma}_{sq} - \alpha\varphi(-D_1^{(sph)}(b_k(\hat{\p},\hat{\q})),\c). \nonumber \\
  \label{eq:negprac18}
\end{eqnarray}
The following theorem summarizes the above results.

\begin{theorem}
  \label{thme:negthmprac1}
  Assume the complete sfl RDT setup of \cite{Stojnicsflgscompyx23}. Consider large $n$ linear regime with $\alpha=\lim_{n\rightarrow\infty} \frac{m}{n}$ and let $\varphi(\cdot)$ and $\bar{\psi}(\cdot)$ be as given in (\ref{eq:prac2}) and (\ref{eq:prac13}), respectively. Also, let the ``fixed'' parts of $\hat{\p}$, $\hat{\q}$, and $\hat{\c}$ satisfy $\hat{\p}_1\rightarrow 1$, $\hat{\q}_1\rightarrow 1$, $\hat{\c}_1\rightarrow 1$, $\hat{\p}_{r+1}=\hat{\q}_{r+1}=\hat{\c}_{r+1}=0$, and let the ``non-fixed'' parts of $\hat{\p}_k$, $\hat{\q}_k$, and $\hat{\c}_k$ ($k\in\{2,3,\dots,r\}$) be the solutions of the following system of equations
  \begin{eqnarray}\label{eq:negthmprac1eq1}
   \frac{d \bar{\psi}_{rd}(\p,\q,\c,\gamma_{sq},1,1,-1)}{d\p} =  0 \nonumber \\
   \frac{d \bar{\psi}_{rd}(\p,\q,\c,\gamma_{sq},1,1,-1)}{d\q} =  0 \nonumber \\
   \frac{d \bar{\psi}_{rd}(\p,\q,\c,\gamma_{sq},1,1,-1)}{d\c} =  0 \nonumber \\
   \frac{d \bar{\psi}_{rd}(\p,\q,\c,\gamma_{sq},1,1,-1)}{d\gamma_{sq}} =  0,
 \end{eqnarray}
Moreover, let $c_k(\hat{\p},\hat{\q})$ and $b_k(\hat{\p},\hat{\q})$ be as in (\ref{eq:prac17}). Then
 \begin{eqnarray}
-f_{sq}^{-}(\infty)
& = &     \frac{1}{2}    \sum_{k=2}^{r+1}\Bigg(\Bigg.
   \hat{\p}_{k-1}\hat{\q}_{k-1}
   -\hat{\p}_{k}\hat{\q}_{k}
  \Bigg.\Bigg)
\hat{\c}_k
  - \varphi(D_1^{(bin)}(c_k(\hat{\p},\hat{\q})),\hat{\c}) + \hat{\gamma}_{sq} - \alpha\varphi(-D_1^{(sph)}(b_k(\hat{\p},\hat{\q})),\hat{\c}). \nonumber \\
  \label{eq:negthmprac1eq2}
\end{eqnarray}
\end{theorem}
\begin{proof}
Analogously to Theorem \ref{thme:thmprac1}, follows from the previous discussion, Theorem \ref{thm:thmsflrdt1}, Corollary \ref{cor:cor1}, and the sfl RDT machinery presented in \cite{Stojnicnflgscompyx23,Stojnicsflgscompyx23,Stojnicflrdt23}.
\end{proof}

\subsubsection{Numerical evaluations}
\label{nuemrical}

The same way results of Theorem \ref{thme:thmprac1} were utilized in the previous section for the numerical evaluations regarding the positive model, we, here, utilize results of Theorem \ref{thme:negthmprac1} to conduct the corresponding  numerical evaluations for the negative model. To ease the exposition, we parallel the positive model evaluations as much as possible. Several analytical results that were obtained earlier for the positive model have their negative model  counterparts and we state them below as well. As earlier, we start with $r=1$ and proceed inductively while, on occasion, specializing the considerations to the square case ($\alpha=1$) to enable the numerical values concreteness.

\underline{1) \textbf{\emph{$r=1$ -- first level of lifting:}}} For $r=1$ one has that $\hat{\p}_1\rightarrow 1$ and $\hat{\q}_1\rightarrow 1$ which together with $\hat{\p}_{r+1}=\hat{\p}_{2}=\hat{\q}_{r+1}=\hat{\q}_{2}=0$, and $\hat{\c}_{2}\rightarrow 0$ gives
\begin{align}\label{eq:negprac19}
    \bar{\psi}_{rd}(\hat{\p},\hat{\q},\hat{\c},\gamma_{sq},1,1,-1)   & =   \frac{1}{2}
\c_2
  - \frac{1}{\c_2}\log\lp \mE_{{\mathcal U}_2} e^{\c_2|\sqrt{1-0}\h_1^{(2)} |}\rp +\gamma_{sq}
- \alpha\frac{1}{\c_2}\log\lp \mE_{{\mathcal U}_2} e^{-\c_2\frac{(\sqrt{1-0}\u_1^{(2,2)})^2}{4\gamma_{sq}}}\rp \nonumber \\
& \rightarrow
  - \frac{1}{\c_2}\log\lp 1+ \mE_{{\mathcal U}_2} \c_2|\sqrt{1-0}\h_1^{(2)} |\rp +\gamma_{sq}
- \alpha\frac{1}{\c_2}\log\lp 1- \mE_{{\mathcal U}_2} \c_2\frac{(\sqrt{1-0}\u_1^{(2,2)})^2}{4\gamma_{sq}}\rp \nonumber \\
& \rightarrow
  - \frac{1}{\c_2}\log\lp 1+ \c_2\sqrt{\frac{2}{\pi}}\rp +\gamma_{sq}
- \alpha\frac{1}{\c_2}\log\lp 1- \frac{\c_2}{4\gamma_{sq}}\rp \nonumber \\
& \rightarrow
 - \sqrt{\frac{2}{\pi}}+\gamma_{sq}
+  \frac{\alpha}{4\gamma_{sq}}.
  \end{align}
One then easily finds $\hat{\gamma}_{sq}=\frac{\sqrt{\alpha}}{2}$ and
\begin{align}\label{eq:negprac20}
 - f^{-,1}_{sq}(\infty)=\bar{\psi}_{rd}(\hat{\p},\hat{\q},\hat{\c},\hat{\gamma}_{sq},1,1,-1)   & =
  - \sqrt{\frac{2}{\pi}}+\sqrt{\alpha},
  \end{align}
which for $\alpha=1$ becomes
\begin{align}\label{eq:negprac21}
 - f^{-,1}_{sq}(\infty) =
  - \sqrt{\frac{2}{\pi}}+1\rightarrow \bl{\mathbf{0.2021}}.
  \end{align}

\underline{2) \textbf{\emph{$r=2$ -- second level of lifting:}}} Paralleling again the presentation of the positive model, we split the second level of lifting into two subparts: (i) \emph{partial} second level of lifting; and (ii) \emph{full} second level of lifting.

For the \emph{partial} part, we have, when $r=2$, that $\hat{\p}_1\rightarrow 1$ and $\hat{\q}_1\rightarrow 1$, $\hat{\p}_{2}=\hat{\q}_{2}=0$, and $\hat{\p}_{r+1}=\hat{\p}_{3}=\hat{\q}_{r+1}=\hat{\q}_{3}=0$ but in general  $\hat{\c}_{2}\neq 0$. As above, one again has
\begin{align}\label{eq:negprac22}
    \bar{\psi}_{rd}(\hat{\p},\hat{\q},\c,\gamma_{sq},1,1,-1)   & =   \frac{1}{2}
\c_2
  - \frac{1}{\c_2}\log\lp \mE_{{\mathcal U}_2} e^{\c_2|\sqrt{1-0}\h_1^{(2)} |}\rp + \gamma_{sq}
- \alpha\frac{1}{\c_2}\log\lp \mE_{{\mathcal U}_2} e^{-\c_2\frac{(\sqrt{1-0}\u_1^{(2,2)})^2}{4\gamma_{sq}}}\rp.
   \end{align}
After taking the derivatives with respect to $\c_2$ and $\gamma_{sq}$, one, for $\alpha=1$, finds $\hat{\gamma}_{sq}=0.1654$ and $\hat{\c}_2=2.6916$, which gives
\begin{align}\label{eq:negprac23}
\hspace{-2in}(\mbox{\emph{partial} second level:}) \qquad \qquad   -f^{-,2}_{sq}(\infty) \rightarrow \bl{\mathbf{0.3202}}.
  \end{align}

For the \emph{full} part, the setup is basically the same as above with the sole exception that now, in general, (in addition to $\hat{\c}_{2}\neq 0$) one also has $\p_2\neq0$ and $\q_2\neq0$. Writing as above, we then find
\begin{align}\label{eq:negprac24}
    \bar{\psi}_{rd}(\p,\q,\c,\gamma_{sq},1,1,-1)   & =   \frac{1}{2}
(1-\p_2\q_2)\c_2
  - \frac{1}{\c_2}\mE_{{\mathcal U}_3}\log\lp \mE_{{\mathcal U}_2} e^{\c_2|\sqrt{1-\q_2}\h_1^{(2)} +\sqrt{\q_2}\h_1^{(3)} |}\rp \nonumber \\
& \quad    + \gamma_{sq}
 -\alpha\frac{1}{\c_2}\mE_{{\mathcal U}_3} \log\lp \mE_{{\mathcal U}_2} e^{-\c_2\frac{(\sqrt{1-\p_2}\u_i^{(2,2)}+\sqrt{\p_2}\u_i^{(2,3)})^2}{4\gamma_{sq}}}\rp \nonumber \\
& =   \frac{1}{2}
(1-\p_2\q_2)\c_2
   -\frac{1}{\c_2}\mE_{{\mathcal U}_3}\log\lp \mE_{{\mathcal U}_2} e^{\c_2|\sqrt{1-\q_2}\h_1^{(2)} +\sqrt{\q_2}\h_1^{(3)} |}\rp \nonumber \\
   & \quad    -(- \gamma_{sq})
- \alpha \lp -\frac{1}{2\c_2} \log \lp \frac{-2\gamma_{sq}-\c_2(1-\p_2)}{-2\gamma_{sq}} \rp  +  \frac{\p_2}{2(-2\gamma_{sq}-\c_2(1-\p_2))}   \rp \nonumber \\
& = -\bar{\psi}_{rd}(\p,\q,\c,-\gamma_{sq},1,1,1).
   \end{align}
After taking the derivatives one, for $\alpha=1$, finds $\hat{\gamma}_{sq}=0.1654$, $\p_2=0$, $\q_2=0$, and $\hat{\c}_2=2.6916$, which gives
\begin{align}\label{eq:negprac25}
\hspace{-2in}(\mbox{\emph{full} second level:}) \qquad \qquad  -f^{-,2}_{sq}(\infty) \rightarrow \bl{\mathbf{0.3202}}.
  \end{align}
In other words, we observe a remarkable \emph{indistinguishability between partial and full lifting} phenomenon.

\vspace{0.1in}

\underline{3) \textbf{\emph{$r=3$ -- third level of lifting:}}} As for the positive model, we  immediately  consider \emph{full} third level of lifting and, for $r=3$, have that $\hat{\p}_1\rightarrow 1$ and $\hat{\q}_1\rightarrow 1$  as well as  $\hat{\p}_{r+1}=\hat{\p}_{4}=\hat{\q}_{r+1}=\hat{\q}_{4}=0$. As above, one again has
 \begin{align}\label{eq:negprac26}
    \bar{\psi}_{rd}(\p,\q,\c,\gamma_{sq},1,1,-1)   & =   \frac{1}{2}
(1-\p_2\q_2)\c_2 + \frac{1}{2}
(\p_2\q_2-\p_3\q_3)\c_3 \nonumber \\
&\quad  - \frac{1}{\c_3}\mE_{{\mathcal U}_4}\log\lp \mE_{{\mathcal U}_3} \lp  \mE_{{\mathcal U}_2} e^{\c_2|\sqrt{1-\q_2}\h_1^{(2)} +\sqrt{\q_2-\q_3}\h_1^{(3)}+\sqrt{\q_3}\h_1^{(4)}|} \rp^{\frac{\c_3}{\c_2}} \rp \nonumber \\
& \quad    + \gamma_{sq}
 -\alpha\frac{1}{\c_3}\mE_{{\mathcal U}_4} \log\lp \mE_{{\mathcal U}_4} \lp \mE_{{\mathcal U}_2} e^{-\c_2\frac{(\sqrt{1-\p_2}\u_i^{(2,2)}+\sqrt{\p_2-\p_3}\u_i^{(2,3)}+\sqrt{\p_3}\u_i^{(2,4)})^2}{4\gamma_{sq}}}\rp^{\frac{\c_3}{\c_2}} \rp \nonumber \\
   & =   \frac{1}{2}
(1-\p_2\q_2)\c_2 + \frac{1}{2}
(\p_2\q_2-\p_3\q_3)\c_3 \nonumber \\
&\quad  - \frac{1}{\c_3}\mE_{{\mathcal U}_4}\log\lp \mE_{{\mathcal U}_3} \lp  \mE_{{\mathcal U}_2} e^{\c_2|\sqrt{1-\q_2}\h_1^{(2)} +\sqrt{\q_2-\q_3}\h_1^{(3)}+\sqrt{\q_3}\h_1^{(4)}|} \rp^{\frac{\c_3}{\c_2}} \rp \nonumber \\
& \quad    -(- \gamma_{sq})
- \alpha
\Bigg(\Bigg. -\frac{1}{2\c_2} \log \lp \frac{-2\gamma_{sq}-\c_2(1-\p_2)}{-2\gamma_{sq}} \rp \nonumber \\
& \quad -\frac{1}{2\c_3} \log \lp \frac{-2\gamma_{sq}-\c_2(1-\p_2)-\c_3(\p_2-\p_3)}{-2\gamma_{sq}-\c_2(1-\p_2)} \rp  \nonumber \\
& \quad +  \frac{\p_3}{2(-2\gamma_{sq}-\c_2(1-\p_2)-\c_3(\p_2-\p_3))}   \Bigg.\Bigg)\nonumber \\
& = -\bar{\psi}_{rd}(\p,\q,\c,-\gamma_{sq},1,1,1).
   \end{align}
Taking the derivatives we obtain (for $\alpha=1$) $\hat{\gamma}_{sq}=0.1748$, $\hat{\p}_3=\hat{\q}_3=0$, $\hat{\p}_2=0.5722$, $\hat{q1}_2=0.0599$, $\hat{\c}_2=10.54$, and $\hat{\c}_1=2.2264$ and
\begin{align}\label{eq:negprac27}
\hspace{-2in}(\mbox{\emph{partial/full} third level:}) \qquad \qquad -f^{-,3}_{sq}(\infty) \rightarrow \bl{\mathbf{0.3272}}.
  \end{align}
We again observe the remarkable \emph{indistinguishability between partial and full lifting} phenomenon.

\underline{4) \textbf{General \emph{$r$-th level of lifting:}}}  For the completeness, we state the form of $\bar{\psi}_{rd}(\p,\q,\c,\gamma_{sq},1,1,1) $ obtained for general $r$. Following the above procedures, we basically have
 \begin{align}\label{eq:negprac29}
   \bar{\psi}_{rd}(\p,\q,\c,\gamma_{sq},1,1,-1)   & =   \frac{1}{2}
\sum_{k=2}^{r+1}(\p_{k-1}\q_{k-1}-\p_k\q_k)\c_k   \nonumber \\
&\quad  - \frac{1}{\c_r}\mE_{{\mathcal U}_{r+1}}\log\lp \mE_{{\mathcal U}_r} \lp \dots \lp \mE_{{\mathcal U}_3} \lp  \mE_{{\mathcal U}_2} e^{\c_2|\sum_{k=2}^{r+1}\sqrt{\q_{k-1}-\q_k}\h_1^{(k)} |} \rp^{\frac{\c_3}{\c_2}} \rp^{\frac{\c_4}{\c_{3}}} \dots \rp^{\frac{\c_r}{\c_{r-1}}} \rp
\nonumber \\
 & \quad    - (- \gamma_{sq})
- \alpha
\Bigg(\Bigg. -\sum_{k=2}^{r}\frac{1}{2\c_k} \log \lp \frac{\Theta_{k}}{\Theta_{k-1}}\rp +\frac{\p_r}{2\Theta_{r}}  \Bigg.\Bigg)\nonumber \\
& = -\bar{\psi}_{rd}(\p,\q,\c,-\gamma_{sq},1,1,1),
    \end{align}
where
\begin{eqnarray}\label{eq:negprac30}
 \Theta_{1} & = & -2\gamma_{sq} \nonumber \\
 \Theta_{k} & = & \Theta_{k-1}-\c_k(\p_{k-1}-\p_k), k\in\{2,3,\dots,r\}.
 \end{eqnarray}
The above characterization of $\bar{\psi}(\cdot)$ gives explicit dependence on all key parameters $\p$, $\q$, $\m$, and $\gamma_{sq}$ and consequently  enables direct computations of all needed derivatives for any $r\in\mN$. In addition to the above presented numerical evaluations for the first three levels, we also conducted the numerical evaluations for the fourth (partial) level $r=4$  and obtained
\begin{align}\label{eq:negprac31}
\hspace{-2in}(\mbox{\emph{partial} fourth level:}) \qquad \qquad -f^{-,4}_{sq}(\infty) \rightarrow \bl{\mathbf{0.3279}}.
  \end{align}
All the relevant quantities for the fourth (partial)  level of lifting, together with the first, second, and third  level corresponding ones are systematically shown in Table \ref{tab:NEG2rsbunifiedsqrtpos}. Observing the results from  Table \ref{tab:NEG2rsbunifiedsqrtpos} more closely, we once again note the rapid convergence of the entire lifting mechanism with the fourth decimal convergence of the ground state energy already achieved on the fourth level of lifting.

\begin{table}[h]
\caption{$r$-sfl RDT parameters; \underline{negative} \emph{square-root} Hopfield model; $\alpha=1$; $\hat{\p}_1,\hat{\q}_1,\hat{\c}_1\rightarrow 1$; $n,\beta\rightarrow\infty$}\vspace{.1in}
\centering
\def\arraystretch{1.1}
\begin{tabular}{||l||c||c|c||c|c||c|c|c||c||}\hline\hline
 \hspace{-0in}$r$-sfl RDT                                             & $\hat{\gamma}_{sq}$    & $\hat{\p}_3$     & $\hat{\p}_2$        & $\hat{\q}_3$ & $\hat{\q}_2$ &  $\hat{\c}_4$  &  $\hat{\c}_3$   & $\hat{\c}_2$   & $-f_{sq}^{-,r}$  \\ \hline\hline
$\mathbf{1}$ (partial/full)                                      & $0.5$ & $0$ & $0$  & $0$ & $0$ &  $\rightarrow 0$ &  $\rightarrow 0$
 &  $\rightarrow 0$  & \bl{$\mathbf{0.2021}$} \\ \hline\hline
 $\mathbf{2}$ (partial/full)                                      & $0.1654$ & $0$ & $0$  & $0$ & $0$ &  $\rightarrow 0$ &  $\rightarrow 0$
 &  $2.6916$  & \bl{$\mathbf{0.3202}$} \\ \hline\hline
 $\mathbf{3}$ (partial/full)                                      & $0.1748$ & $0$ & $0.5722$  & $0$ & $0.0599$  & $\rightarrow 0$
 & $10.54$ & $2.2264$  & \bl{$\mathbf{0.3272}$} \\ \hline\hline
 $\mathbf{4}$ (partial)                                      & $0.1766$ & $.3639$ & $0.6836$  & $0.0107$ & $0.1282$ &  $27.98$
 & $5.07$ &  $2.1306$  & \bl{$\mathbf{0.3279}$} \\ \hline\hline
\end{tabular}
\label{tab:NEG2rsbunifiedsqrtpos}
\end{table}

We also add that all considerations regarding more efficient numerical evaluations between (\ref{eq:thm3eq2})-(\ref{eq:thm3eq8a1}) are trivial to repeat. The only thing that needs to be adjusted is $s$. Instead of $s=1$ now $s=-1$. The final results are exactly as those stated in (\ref{eq:thm3eq7})-(\ref{eq:thm3eq8a1}). The only tiny difference is that now $\hat{\gamma}_{sq}$ from Table \ref{tab:NEG2rsbunifiedsqrtpos} should be taken with the negative sign in these equations.

The above mentioned \emph{indistinguishability between partial and full lifting} phenomenon clearly visible in Table \ref{tab:NEG2rsbunifiedsqrtpos} disappears as $\alpha$ increases. For the completeness, we provide in Table \ref{tab:NEG2rsbunifiedsqrtpos2} a scenario ($\alpha=3.5$) where one observes the behavior that fully resembles the one seen in the positive model with a clear distinction between partial and full lifting on both second and third level. The convergence on the fourth decimal is achieved now on the third level which also parallels the behavior of the positive models.

\begin{table}[h]
\caption{$r$-sfl RDT parameters; \underline{negative} \emph{square-root} Hopfield model; $\alpha=3.5$; $\hat{\c}_1\rightarrow 1$; $n,\beta\rightarrow\infty$}\vspace{.1in}
\centering
\def\arraystretch{1.2}
\begin{tabular}{||l||c||c|c|c||c|c|c||c|c||c||}\hline\hline
 \hspace{-0in}$r$-sfl RDT                                             & $\hat{\gamma}_{sq}$    & $\hat{\p}_3$ & $\hat{\p}_2$ & $\hat{\p}_1`$        & $\hat{\q}_3$ & $\hat{\q}_2$  & $\hat{\q}_1$ &  $\hat{\c}_3$   & $\hat{\c}_2$    & $f_{sq}^{-,r}$  \\ \hline\hline
$\mathbf{1}$ (full)                                      & $0.9354$ & $0$ & $0$  & $\rightarrow 1$  & $0$ & $0$ & $\rightarrow 1$ &  $\rightarrow 0$
 &  $\rightarrow 0$  & \bl{$\mathbf{1.0729}$} \\ \hline\hline
 $\mathbf{2}$ (partial)                                      & $0.6366$ & $0$ & $0$ & $\rightarrow 1$ & $0$ & $0$ & $\rightarrow 1$ &  $\rightarrow 0$
 &  $1.4759$   & \bl{$\mathbf{1.1288}$} \\ \hline
  $\mathbf{2}$ (full)                                      & $0.6234$ & $0$ & $0.3416$ & $\rightarrow 1$ & $0$ & $0.2154$ & $\rightarrow 1$ &  $\rightarrow 0$
 &  $1.6843$   & \bl{$\mathbf{1.1301}$}  \\ \hline\hline
 $\mathbf{3}$ (partial)                                      & $0.6030$ & $0$ & $0.7921$  & $\rightarrow 1$ & $0$ & $0.6399$ & $\rightarrow 1$ &  $1.15$
 &  $2.26$   & \bl{$\mathbf{1.1309}$}  \\ \hline
  $\mathbf{3}$ (full)                                      & $0.5940$ & $0.2464$ & $0.8935$  & $\rightarrow 1$ & $0.1508$ & $0.7887$ & $\rightarrow 1$ &  $1.40$
 &  $2.79$   & \bl{$\mathbf{1.1312}$}  \\ \hline\hline
\end{tabular}
\label{tab:NEG2rsbunifiedsqrtpos2}
\end{table}

\subsubsection{Modulo-$\m$ sfl RDT}
\label{sec:posmodm}

As in the case of the positive model, one can rely on the modulo-$m$ sfl RDT frame of \cite{Stojnicsflgscompyx23} for the negative model as well. Instead of Theorem \ref{thme:negthmprac1}, we then have the following theorem.

\begin{theorem}
  \label{thme:negthmprac2}
Assume the setup of Theorem \ref{thme:negthmprac1} and instead of the complete, assume the modulo-$\m$ sfl RDT setup of \cite{Stojnicsflgscompyx23}. Let
 \begin{align}\label{eq:negthmprac2eq1}
   \bar{\psi}_{rd}(\p,\q,\c,\gamma_{sq},1,1,-1)   & =   \frac{1}{2}
\sum_{k=2}^{r+1}(\p_{k-1}\q_{k-1}-\p_k\q_k)\c_k   \nonumber \\
&\quad  - \frac{1}{\c_r}\mE_{{\mathcal U}_{r+1}}\log\lp \mE_{{\mathcal U}_r} \lp \dots \lp \mE_{{\mathcal U}_3} \lp  \mE_{{\mathcal U}_2} e^{\c_2|\sum_{k=2}^{r+1}\sqrt{\q_{k-1}-\q_k}\h_1^{(k)} |} \rp^{\frac{\c_3}{\c_2}} \rp^{\frac{\c_4}{\c_{3}}} \dots \rp^{\frac{\c_r}{\c_{r-1}}} \rp
\nonumber \\
 & \quad    -(- \gamma_{sq})
- \alpha
\Bigg(\Bigg. -\sum_{k=2}^{r}\frac{1}{2\c_k} \log \lp \frac{\Theta_{k}}{\Theta_{k-1}}\rp +\frac{\p_r}{2\Theta_{r}}  \Bigg.\Bigg) \nonumber \\
& =    -\bar{\psi}_{rd}(\p,\q,\c,-\gamma_{sq},1,1,1)\nonumber \\
  \Theta_{1} & =  -2\gamma_{sq} \nonumber \\
 \Theta_{k} & =  \Theta_{k-1}-\c_k(\p_{k-1}-\p_k), k\in\{2,3,\dots,r\}.
 \end{align}
Let $\hat{\gamma}_{sq}$ the ``non-fixed" parts of $\hat{\p}$ and $\hat{\q}$ be the solutions of the following systems of equations
\begin{align}\label{eq:negthmprac2eq2}
   \frac{d \bar{\psi}_{rd}(\p,\q,\c,\gamma_{sq},1,1,-1)}{d\p} & = - \frac{d \bar{\psi}_{rd}(\p,\q,\c,-\gamma_{sq},1,1,1)}{d\p}
   = \frac{d \bar{\psi}_{rd}(\p,\q,\c,\gamma_{sq},1,1,1)}{d\p} = 0 \nonumber \\
   \frac{d \bar{\psi}_{rd}(\p,\q,\c,\gamma_{sq},1,1,-1)}{d\q} & =  -\frac{d \bar{\psi}_{rd}(\p,\q,\c,-\gamma_{sq},1,1,1)}{d\q}
   = \frac{d \bar{\psi}_{rd}(\p,\q,\c,\gamma_{sq},1,1,1)}{d\q}= 0 \nonumber \\
    \frac{d \bar{\psi}_{rd}(\p,\q,\c,\gamma_{sq},1,1,-1)}{d\gamma_{sq}} & =   - \frac{d \bar{\psi}_{rd}(\p,\q,\c,-\gamma_{sq},1,1,1)}{d\gamma_{sq}}
    = \frac{d \bar{\psi}_{rd}(\p,\q,\c,\gamma_{sq},1,1,1)}{d\gamma_{sq}}=  0.
 \end{align}
Then
 \begin{eqnarray}
-f_{sq,\m}^{-}(\infty)
& \leq  & \max_{\c} \lp  \bar{\psi}_{rd}(\hat{\p},\hat{\q},\c,\hat{\gamma}_{sq},1,1,-1)\rp
= \min_{\c} \lp - \bar{\psi}_{rd}(\hat{\p},\hat{\q},\c,-\hat{\gamma}_{sq},1,1,1)\rp.
  \label{eq:negthmprac2eq3}
\end{eqnarray}
\end{theorem}
\begin{proof}
Follows from the previous discussion, Theorem \ref{thm:thmsflrdt1}, Corollary \ref{cor:cor1}, and the sfl RDT machinery presented in \cite{Stojnicnflgscompyx23,Stojnicsflgscompyx23,Stojnicflrdt23}.
\end{proof}

We again conducted the numerical evaluations (albeit on a smaller range of $\c$ parameters) using the modulo-$\m$ results of the above theorem and obtained $-f_{sq}^{-,r}(\infty)=-f_{sq,\m}^{-,r}(\infty)$.

\section{Conclusion}
\label{sec:conc}

The famous Hopfield models and their behavior in statistical mediums are the main subject of study in this paper. While structurally somewhat similar to the celebrated Sherrington-Kirkpatrick (SK) models, these models are a much harder mathematical challenge. Early statistical physics considerations \cite{CriAmiGut86,SteKuh94,AmiGutSom87}
  traced a possible path for mathematically rigorous considerations. However, despite the existence of such a, potentially helpful, tool, the progress on the rigourous front   has been pretty much stagnant over the last several decades. Namely, the mathematically rigorous treatments, (see, e.g.,  \cite{PasShchTir94,ShchTir93,BarGenGueTan10,Tal98,Zhao11}), typically relate to the so-called high-temperature ($T$) regime (low inverse temperature $\beta=\frac{1}{T}$), where the so-called replica symmetric behavior is in place. On the other hand, of a prevalent interest in many mathematical areas (including optimization, algorithms, machine learning and neural networks) are precisely their counterparts, i.e., the regimes where the symmetric behavior is not in place. Moreover, the hardest of them all, the so-called ground state, zero temperature, regime is typically the most sought after one.

We start by recognizing that, due to measure concentrations, one can, instead of the standard Hopfield models and associated ground state energies, study their so-called \emph{square-root} variants. We then distinguish two important classes of the models: (i) the \emph{positive}; and (ii) the \emph{negative} ones. Utilizing the random duality theory (RDT) principles, \cite{StojnicHopBnds10} established the upper bounds for the positive and the lower bounds for the negative models which both turned out to precisely match the corresponding replica symmetry based predictions. On the other hand, \cite{StojnicMoreSophHopBnds10} then went a bit further, used a \emph{lifted} RDT variant, and lowered the \cite{StojnicHopBnds10}'s bounds on the positive and lifted the \cite{StojnicHopBnds10}'s bounds on the negative models, thereby showing that the replica symmetry produces a strict (but \emph{not} tight)  bounds and ensuring that the symmetry \emph{must} be broken.

We, here, continue the line of study of  \cite{StojnicHopBnds10,StojnicMoreSophHopBnds10}, where the Hopfield models are connected to the behaviour of the bilinearly indexed (bli) random processes. However, we, now, rely on a strong recent progress made in studying bli processes in \cite{Stojnicnflgscompyx23,Stojnicsflgscompyx23} and on the main foundational principles of the \emph{stationarized fully lifted} (sfl) RDT established in \cite{Stojnicflrdt23} based on that progress. After recognizing that the statistical variants of the Hopfield models are indeed directly related to the bli processes, we show that the sfl RDT principles can be used to characterize their statistical behavior. As the obtained analytical characterizations rely on a nested concept of lifting, they are fully operational only if one can successfully conduct underlying numerical evaluations. After conducting such evaluations for the ground state energies of both positive and negative models, we observe an incredible level of convergence of the employed sfl RDT mechanism. Namely, for the so-called square case ($m=n$), the fourth decimal precision is achieved already on the third (second non-trivial) level (3-sfl RDT) for the positive case and on the fourth (third non-trivial) level (4-sfl RDT) for the corresponding negative one. In particular, we obtain the scaled ground state free energy $\approx 1.7788$ for the positive and $\approx 0.3279$  for the negative model. We also observe a remarkable structural difference in the lifting form of both models. It turns out that the positive model exhibits the so-called \emph{full} lifting for any of the evaluated levels. On the other hand, the negative model exhibits exclusively the \emph{partial} lifting. In other words, for the negative models, we uncover the existence of the ``\emph{indistinguishability between the partial and full lifting}'' phenomenon. Even more remarkably, this phenomenon disappears as $\alpha$ increases. While we uncover both the existence and the disappearance of this phenomenon as completely mathematical objects, it will be more than interesting to see if a statistical physics explanation can be associated with it.

It is not that difficult to see that the introduced methodology easily allows for various extensions and generalizations. For example, all those discussed in \cite{Stojnicnflgscompyx23,Stojnicsflgscompyx23,Stojnicflrdt23},  including those related to the  asymmetric Little models
(see, e.g.,  \cite{BruParRit92,Little74,BarGenGue11bip,CabMarPaoPar88,AmiGutSom85,StojnicAsymmLittBnds11}), various \emph{random feasibility problems} (rfps) such as  \emph{spherical} perceptrons   (see, e.g., \cite{FPSUZ17,FraHwaUrb19,FraPar16,FraSclUrb19,FraSclUrb20,AlaSel20,StojnicGardGen13,StojnicGardSphErr13,StojnicGardSphNeg13,GarDer88,Gar88,Schlafli,Cover65,Winder,Winder61,Wendel,Cameron60,Joseph60,BalVen87,Ven86,SchTir02,SchTir03}), \emph{binary} perceptrons (see, e.g., \cite{StojnicGardGen13,GarDer88,Gar88,StojnicDiscPercp13,KraMez89,GutSte90,KimRoc98}), the \emph{sectional} compressed sensing $\ell_1$ phase transitions (see, e.g., \cite{StojnicCSetam09,StojnicLiftStrSec13}), and many others are immediately possible. Similarly to the problems that we considered here, such extensions require a bit of technical juggling that is problem specific and we discuss them, together with all underlying intricacies, in separate papers.

As already mentioned in \cite{Stojnicflrdt23}, the sfl RDT considerations, and therefore the ones presented here, do not require the standard Gaussianity  assumption of the random primals. Relying on the central limit theorem and, in particular, its Lindeberg variant \cite{Lindeberg22}, the obtained results can quickly be extended to a wide range of different statistics. Results of \cite{Chatterjee06} are a particularly elegant approach that enables such a quick extension.

\begin{singlespace}
\bibliographystyle{plain}
\bibliography{nflgscompyxRefs}
\end{singlespace}

\end{document}